%% file: main.tex
\documentclass[10pt,a4paper,reqno]{amsart}
\usepackage[utf8]{inputenc}
\usepackage[T1]{fontenc}
\usepackage[foot]{amsaddr}
\usepackage{amsmath}
\usepackage{xcolor}
\usepackage{soul}
\usepackage{graphicx}
\usepackage{wrapfig}
\usepackage{subcaption}
\usepackage[nopar]{lipsum}
\usepackage{amsthm}
\usepackage{amssymb}
\usepackage{euscript}
\usepackage{enumitem}
\usepackage{mathtools}
\usepackage{soul}
\usepackage{amscd}
\usepackage{float}
\usepackage[pagebackref, hypertexnames=false, colorlinks, citecolor=red, linkcolor=blue, urlcolor=red]{hyperref}
\usepackage{cleveref}
\usepackage{geometry}
\usepackage{pgf,tikz,pgfplots}
\usepackage{tikzit}
\usepackage{adjustbox}
\usepackage{setspace}
\usepackage[numbers,sort&compress]{natbib}

\allowdisplaybreaks[4]
\usetikzlibrary{calc}
\usetikzlibrary{patterns}
\pgfplotsset{compat=1.14}
\input{black.tikzstyles}
\numberwithin{equation}{section}
\newtheorem{thm}{Theorem}[section]
\newtheorem{prop}[thm]{Proposition}

\newtheorem{cor}[thm]{Corollary}
\newtheorem{exm}[thm]{Example}
\newtheorem{df}[thm]{Definition}
\newtheorem{rem}[thm]{Remark}
\newtheorem{nt}[thm]{Note}
\makeatletter
\@namedef{subjclassname@1991}{Subject}
\@namedef{subjclassname@2000}{Subject}
\@namedef{subjclassname@2010}{Subject}
\@namedef{subjclassname@2020}{Subject}
\makeatother
\begin{document}
	\title{On Some General Operators of  Hypergraphs}
	\email[ Banerjee]{\textit {{\scriptsize anirban.banerjee@iiserkol.ac.in}}}
	
	\author[Banerjee]{Anirban Banerjee } 
	\address[Banerjee]{Department of Mathematics and Statistics, Indian Institute of Science Education and Research Kolkata, Mohanpur-741246, India}
		\email[Parui ]{\textit {{\scriptsize  samironparui@gmail.com}}}
	
	\author[ Parui]{ Samiron Parui} 
	\address[Parui]{Department of Mathematics and Statistics, Indian Institute of Science Education and Research Kolkata, Mohanpur-741246, India}
	\date{\today}
	

	\maketitle
	\begin{abstract}
Here we introduce connectivity operators, namely, diffusion operators, general Laplacian operators, and general adjacency operators for hypergraphs. These operators are generalisations of some conventional notions of apparently different connectivity matrices associated with hypergraphs. 
In fact, we introduce here a unified framework for studying different variations of the connectivity operators associated with hypergraphs at the same time. 
Eigenvalues and corresponding eigenspaces of the general connectivity operators associated with some classes of hypergraphs are computed. Applications such as random walks on hypergraphs, dynamical networks, and disease transmission on hypergraphs are studied in the perspective of our newly introduced operators. We also derive spectral bounds for the weak connectivity number, degree of vertices, maximum cut, bipartition width, and isoperimetric constant of hypergraphs.

	\end{abstract}
	\hspace{10pt}
	
	\section{Introduction}
\textit{Graph} is a well studied and widely used notion in the realm of mathematics.  \textit{Hypergraph}, a generalization of a graph, is also explored extensively. A hypergraph $G$ is an order pair of sets $(V,E)$, where 
     $V (\neq\emptyset )$ is the set of vertices, and any element $e \in E$, called a hyperedge of $G$, is 
     a nonempty subset of $V$. 
     Thus a graph is a special case of hypergraph where $|e|=2$ for all $e\in E$.
    A singleton hyperedge is said to be a loop. In our work we consider hypergraphs without any loop, that is, $|e|\ge 2$, $\forall e \in E$.
In this work, we are going to introduce some general notions of connectivity operators associated with hypergraphs. These notions are so exhaustive that can incorporate multiple conventional notions of connectivity matrices of hypergraphs. We provide constructions to determine eigenvectors and their eigenspaces of general connectivity operators associated with a class of hypergraphs. Since our approach pivot around some real-valued functions on the vertex set, using our methods, just looking at the structure of the hypergraphs one can determine the eigenvalues and their eigenspaces of the above-mentioned operators for some classes of hypergraphs.

The last decade witnessed a revolution in hypergraph theory when different tensors or hypermatrices associated with hypergraphs are studied extensively in \cite{qi2017tensor, MR3598572,robeva2019duality,zhang2017some} and references therein.
Despite promising progress, some aspects of spectral graph theory cannot be generalised to spectral hypergraph theory using tensors.
The high computational complexity of the tensors associated with hypergraphs is another challenge in studying many spectral aspects of hypergraphs.
Most tensor-related problems are NP-hard, as shown in \cite{MR3144915}.
The alternative method for studying a hypergraph is to examine the underlying graph with appropriate weights. 
Different properties of a hypergraph are studied in terms of the spectra of different connectivity matrices associated with the underlying weighted graph of the hypergraph, see \cite{bretto2013hypergraph,rodriguez2003Laplacian,rodriguez2009Laplacian,MR4208993}. 
Since many significant properties (including the connectivity among the vertices) of a hypergraph are encrypted in the spectra of these matrices, they are generally referred to as the connectivity matrices of the hypergraph.

In this article, we introduce some linear operators associated with a hypergraph which are generalization of some conventional notion of apparently different connectivity matrices associated with that hypergraph. In fact, here we attempt to unify some apparently different but similar concepts of connectivity matrices. Moreover, keeping in mind some applications of their special cases, we can predict some possible real-world applications of our introduced operators. 
 Now we summarise the content of this article in brief.
  \Cref{basics} is devoted to introducing the general diffusion operator. In \Cref{prilims} we define some preliminary notions that we are going to use throughout the article. The general diffusion operator, one pivotal notion of this article, is introduced in \Cref{gen-diff-exm}. Some stimulating examples are included in this section.
The spectra of the diffusion operator are studied in \Cref{eigdiffprop}. We provide eigenvalues of diffusion operator of hypergraph having some special property in \Cref{cute1}, \Cref{cute2}, \Cref{cute3}. We use one of the most natural approaches to analyze the eigenvalues of an operator. We exploit the eigenvectors of the operators in the above-mentioned theorems. 
We provide some results in \Cref{eigdiffprop} which facilitate us to find the eigenvalues and eigenvectors of some classes of diffusion operators of some types of hypergraphs simply from the structure of the hypergraphs.  We calculate the eigenvalues and corresponding eigenspaces of diffusion operators of a class of hypergraphs in \Cref{spectra_ex}. We provide the complete spectra and corresponding eigenspaces of hyperflower hypergraph in \Cref{spectra_ex}.We show that the Laplacian operator is a constant multiple of the diffusion operator. Therefore, one can easily estimate the spectra of the Laplacian operator of hypergraphs from the same of the diffusion operators given in this section. 
In \Cref{bounds}, we investigate the spectral bounds of several hypergraph properties in terms of the spectra of the diffusion operator. 
We also derive spectral bounds for weak connectivity number, degree of vertices, maximum cut, bipartition width, isoperimetric constant of hypergraphs.
 In \Cref{adjacency}, we introduce the general adjacency operator.
 \Cref{nor-lap} is devoted to the normalized Laplacian. Some potential applications of our study are presented in \Cref{app}.
 
 \section{General diffusion operators of a hypergraph}\label{basics}
 Let $\mathbb{R}^{V}$ be the set of all  real-valued functions on the vertex set $V$ and $\mathbb{R}^E$ denote the set of all  real-valued functions on the set of all the hyperedges $E$. Suppose that $\mathbf{1}\in \mathbb{R}^V$ is defined by $\mathbf{1}(v)=1 $ for all $v\in V$. Let $\mathfrak{M}$ be a collection of linear operators on $\mathbb{R}^V$ such that for all $M\in \mathfrak{M}$, $\lim\limits_{t\to\infty}s(t)= c\mathbf{1}$ for some $c\in\mathbb{R}$, and  where $s:\mathbb{R}\to \mathbb{R}^V$ is a solution of the differential equation $\dot{x}(t)=M(x(t))$.
Diffusion processes end up with equality of concentration after the movement of substances from higher concentration to lower concentration. Therefore,  
 $\lim\limits_{t\to\infty}s(t)= c\mathbf{1}$ can be interpreted as a diffusion under the action of the operator $ M$ and we refer any $M\in\mathfrak{M}$ as a diffusion operator. 
This section is devoted to finding a diffusion operator associated with a hypergraph. Now we recall some preliminaries related to hypergraphs.

  The corank, $cr(G)$ and rank, $rk(G)$ of a hypergraph, $G=(V,E)$, is defined by $cr(G)=\min\limits_{e\in E}|e|,$ and $ rk(G)=\max\limits_{e\in E}|e|.$ A hypergraph $G$ is called $m$-uniform hypergraph if $cr(G) = rk(G) = m$.
 Suppose that $v_0,v_l\in V$. A \textit{path $v_0-v_l$ of length $l$ connecting the vertices $v_0$ and $v_l$} in a hypergraph $G=(V,E)$ is an alternating sequence $v_0e_1v_1e_2v_2\ldots v_{l-1}e_lv_l$ of distinct vertices $v_0,v_1,\ldots,v_{l-1},v_l$ and distinct hyperedges $e_1,e_2,\ldots,e_l$, such that, $v_{i-1}, v_i \in e_i$ for all $i= 1, \dots, l$.
 The\textit{ distance, $d(u,v)$, between two vertices $u,v\in V$}  is the minimum among the length of all paths connecting the vertices $u$ and $v$. The \textit{diameter}, $diam(G)$ of a hypergraph $G(V,E)$ is defined by $diam(G)=\max\limits_{u,v\in V}d(u,v)$. 
 An \textit{weighted hypergraph} $G=(V,E,w)$ is a hypergraph with a function $w:E\to \mathbb{R}^{+}$, called the weight of the hyperedges.
 For an weighted hypergraph $G=(V,E,w)$, the \textit{degree of a vertex} $v\in V$ is defined by, $d(v)=\sum_{e\in E_v}w(e)$, where $E_v$ is the collection of all the hyperedges containing the vertex $v$. In \cite{bretto2013hypergraph}, $E_v$ is referred as the star centered in $v$. 
If the hypergraph is unweighted then $w(e)=1$ for all $e\in E$ and then $d(v)=|E_v|$.  

 \subsection{The average operator and general signless Laplacian operator}\label{prilims}
 
 We consider $V$ and $E$ are two finite sets.
Let $\delta_V:V\to\mathbb{R}^+$ and $\delta_E:E\to\mathbb{R}^+$ be two positive real-valued functions on  the vertices and hyperedges, respectively.
We define below inner products on $\mathbb{R}^{V}$ and $\mathbb{R}^E$. 
\begin{df}
\begin{enumerate}
    \item ({Inner product on $\mathbb{R}^{V}$} )
Given $x,y\in$ $\mathbb{R}^{V}$, let 
$$(x,y)_V:=\sum\limits_{v\in V}\delta_V(v)x(v)y(v).$$
\item (Inner product on $\mathbb{R}^{E}$) Given $\beta,\gamma\in$ $\mathbb{R}^{E}$, let 
$$(\beta,\gamma)_E:=\sum\limits_{e\in E}\delta_E(e)\beta(e)\gamma(e).$$
\end{enumerate}
\end{df}
Now we define a function from $\mathbb{R}^{V}$ to $\mathbb{R}^{E}$, which  will produce the average of any given real-valued function on $V$ on a given hyperedge $e$. 
\begin{df}[Average operator ] Given $x\in \mathbb{R}^{V}$, $e \in E$, the function $avg:\mathbb{R}^{V} \to \mathbb{R}^{E}$ is defined by 
$$(avg(x))(e):=\frac{\sum\limits_{v\in e}x(v)}{|e|}.$$ where $|e|$
is the cardinality of $e$.    
\end{df} 

Now we introduce the adjoint of $avg$. 	
\begin{df}[Adjoint of the average operator]
Given $\beta \in \mathbb{R}^{E} $, $ v\in V$ the function $avg^*:\mathbb{R}^{E}\to \mathbb{R}^{V}$ is defined by
$$(avg^*(\beta))(v):=\sum\limits_{e\in E_v}\frac{\beta(e)}{|e|}\frac{\delta_E(e)}{\delta_V(v)}.$$
\end{df}

Now we show that $avg^*:\mathbb{R}^{E}\to \mathbb{R}^{V}$ is the unique choice for being the adjoint of the average operator.
\begin{prop}
For any $x\in \mathbb{R}^V$ and any $\beta \in \mathbb{R}^E$,
$(avg(x),\beta)_E=(x,avg^*(\beta))_V$.
\end{prop}
\begin{proof}
\begin{align*}
    (avg(x),\beta)_E
    &=\sum\limits_{e\in E}\delta_E(e)(avg(x))(e)\beta(e)\\
    &=\sum\limits_{e\in E}\delta_E(e)\beta(e)\frac{\sum\limits_{v\in e}x(v)}{|e|}\\
    &=\sum\limits_{v\in V}\delta_V(v)x(v)\sum\limits_{e\in E_v}\frac{\beta(e)}{|e|}\frac{\delta_E(e)}{\delta_V(v)}\\
    &=(x,avg^*(\beta))_V.
\end{align*}
\end{proof}
Clearly, For all $x\in\mathbb{R}^{V}$ and $v\in V$ the expression of the function,  $avg^*\circ avg:\mathbb{R}^{V}\to \mathbb{R}^{V}$ is $$ (avg^*\circ avg)(x)(v)=\sum\limits_{e\in E_v}\frac{(avg(x))(e)}{|e|}\frac{\delta_E(e)}{\delta_V(v)}.$$ From now onward, we denote the operator  $avg^*\circ avg$ by $\mathcal Q$. Therefore, the operator $\mathcal Q:\mathbb{R}^{V}\to \mathbb{R}^{V}$ is defined by
\begin{equation}\label{Q}
	\mathcal Q(x)(v)= (avg^*\circ avg)(x)(v)=\sum\limits_{e\in E_v}\frac{(avg(x))(e)}{|e|}\frac{\delta_E(e)}{\delta_V(v)}, 
\end{equation}
for any $x\in\mathbb{R}^{V}$ and $v\in V$. 
\begin{rem}\label{remq} Now we have the following observations on $\mathcal Q$.
\begin{itemize}
    \item [(1)] Evidently, $(x,\mathcal{Q}x)_V=(x,(avg^*\circ avg)x)_V=((avg(x),avg(x))_E\ge 0$. Therefore, $\mathcal{Q}$ is a positive semidefinite operator. Moreover, $\mathcal{Q} $ is self-adjoint since $(x,\mathcal{Q}y)_V=((avg(x),avg(y))_E=((avg(y),avg(x))_E=(y,\mathcal{Q}x)_V=(\mathcal{Q}x,y)_V $. 
    \
    \item [(2)] From \cref{Q} we have $ \mathcal{Q}(\mathbf{1})(v)=\sum\limits_{e\in E_v}\frac{(avg(\mathbf{1}))(e)}{|e|}\frac{\delta_E(e)}{\delta_V(v)}=\sum\limits_{e\in E_v}\frac{1}{|e|}\frac{\delta_E(e)}{\delta_V(v)}$. If $\sum\limits_{e\in E_v}\frac{1}{|e|}\frac{\delta_E(e)}{\delta_V(v)}=c$ for all $v\in V$ then $c$ is an eigenvalue with eigenvector $\mathbf{1}$. Moreover, if $\delta_E(e)=|e|$ and $\delta_V(v)=|E_v|$ for all $e\in E$ and $v\in V$, then $ \mathcal Q(\mathbf{1})(v)=\sum\limits_{e\in E_v}(avg(\mathbf{1}))(e)\frac{1}{|E_v|}=1$. Therefore, $1$ is an eigenvalue with eigenvector $\mathbf{1}$.
    \item [(3)] Consider  $\delta_E(e)=|e|^2$ and $\delta_V(v)=1$. Then
    \begin{align*}
        \mathcal Q(x)(v)&=\sum\limits_{e\in E_v}\frac{(avg(x))(e)}{|e|}\frac{\delta_E(e)}{\delta_V(v)} \notag\\&
        =\sum\limits_{e\in E_v}\sum\limits_{u\in e}x(u)\\&
        =\sum\limits_{e\in E}B_{ve}\sum\limits_{u\in V}B_{ue}x(u)=((BB^T)x)(v),
    \end{align*}
    
    where $B=\left(B_{ue}\right)_{u\in V,e\in E}$ and $B_{ue}=1$ if $u\in e$ and otherwise $B_{ue}=0$.
    Therefore, $\mathcal{Q}$ becomes the operator associated with the signless Laplacian matrix $BB^T$, described in \cite[p. 1]{cardoso2019signless}. This motivates us to  refer the operator $\mathcal{Q}$ as the general signless Laplacian of hypergraphs.
\end{itemize}
\end{rem}
\subsection{The general diffusion operator} \label{gen-diff-exm}
We define a function $n\in \mathbb{R}^{V}$ by $$n=\mathcal{Q}(\mathbf{1}).$$

Now we define the general diffusion function $L_G:\mathbb{R}^{V}\to\mathbb{R}^{V}$ by 
$$(L_G(x))(v)=\mathcal{Q}(x)(v)-n(v)x(v),$$ 
for all $x\in \mathbb{R}^{V}$, $v\in V$. 
Now onward, we denote $L_G$  by $L$, when there is no scope of confusion regarding the hypergraph $G$.
\begin{nt}\label{gen} Different concepts of diffusion operators and Laplacian operators are available in the literature \cite{bretto2013hypergraph,rodriguez2003Laplacian,rodriguez2009Laplacian,MR4208993,banerjee2020synchronization}. In the following example, we show for the proper choices of $\delta_V$ and $\delta_E$, our notion of diffusion operator of hypergraph coincides with some existing notions of Laplacian operators and diffusion operators of hypergraphs.
\end{nt}
\begin{exm}\label{ex-diff-hy}
\begin{enumerate}\label{ex-diff}
    \item \label{L1}
If we take $\delta_V(v)=1$ and $\delta_E(e)=|e|^2$, then the operator $L$ becomes the negative of the Laplacian matrix, described in \cite{rodriguez2003Laplacian,rodriguez2009Laplacian}.
\item \label{L2}
If we choose $\delta_V(v)=1$ for all $v\in V$ and $\delta_E(e)=w(e)\frac{|e|^2}{|e|-1} $ then the operator $L$ becomes equal with the diffusion operator described in \cite{banerjee2020synchronization} for weighted hypergraphs. Moreover, for unweighted hypergraphs, i.e., when $w(e)=1$ for all $e\in E$, $L$ becomes the negative of the Laplacian operator defined in \cite{MR4208993}.

\item \label{L3}
When $\delta_V(v)=|E_v|$, and $\delta_E(e)=\frac{|e|^2}{|e|-1}$, $L$ becomes negative of the normalized Laplacian given in \cite{MR4208993}.
\end{enumerate}
\end{exm}

The above examples motivate us to defined the general Laplacian operator $\mathfrak{L}$ associated with hypergraphs as, $\mathfrak{L}=-L$. Since, studying any one of the operators, $L$ and $\mathfrak{L}$, do the same for the other, from now, we focus on $L$. 
\begin{rem}
Any result of this article involving any conditions on $\delta_E$ and $\delta_V$ on the diffusion operator, Laplacian operator, adjacency operator can be converted to a result on the operators given in \cite{MR4208993, banerjee2020synchronization,rodriguez2003Laplacian,rodriguez2009Laplacian} by choosing $\delta_E$, $\delta_V$ accordingly. Similarly, if one choose other $\delta_E$ and $\delta_V$ to incorporate different situations then all the results of this article can be converted to their framework by appropriately choosing $\delta_E$ and $\delta_V$.
\end{rem}
\section{Eigenvalues of the general diffusion operators of hypergraphs}\label{eigdiffprop}
Since the map $f:\mathbb{R}^{V}\to \mathbb{R}^{V}$, defined by $(f(x))(v)=n(v)x(v)$, for all $x\in \mathbb{R}^{V},v\in \mathbb{R}^{V}$, is  self-adjoint, and the operator, $\mathcal{Q}(x)$ is  self-adjoint, thus $L$ is also self-adjoint.
From the definition of $L$, it follows that 
\begin{align}\label{L}
    (Lx)(v)=\sum\limits_{e\in E_v}\frac{\delta_E(e)}{\delta_V(v)}\frac{1}{|e|^2}\sum\limits_{u\in e}(x(u)-x(v)),
\end{align}
for all $x\in \mathbb{R}^V, v\in V$. For each hyperedge $e\in E$, suppose $Q_e$ is the incident matrix of a complete graph $K_{|e|}$ involving all the vertices in $e$ with a fixed orientation. So, $Q_e=\left\{{Q_e}_{rv}\right\}_{r\in \mathbb{N}_{\binom{|e|}{2}}, v\in V}$, where ${Q_e}_{rv}=-1$ if $v$ is the head of the $r$-th edge of the oriented $K_{|e|}$,  ${Q_e}_{rv}=1$ if $v$ is the tail of the $r$-th edge of the oriented $K_{|e|}$, and ${Q_e}_{rv}=0$ otherwise. Here $\mathbb{N}_r$ is the collection of all the natural numbers $\le r$. It is easy to verify that for each $x\in \mathbb{R}^V$, $Lx=-(\sum\limits_{e\in E}\frac{\delta_E(e)}{|e|^2}\Delta_V^{-1}Q_e^tQ_e)x$, where $\Delta_V$ is a diagonal matrix of order $|V|$ such that $\Delta_V(v,v)=\delta_V(v)$ for all $v\in V$. 
\begin{prop}\label{nsdo}
$L$ is negative semidefinite.
\end{prop}
\begin{proof}
For any $x\in \mathbb{R}^{V}$,
\begin{align}
 \label{nsdl}   \notag (L(x),x)_V 
               &=\sum\limits_{v\in V}\delta_V(v)[\mathcal{Q}(x)(v)-n(v)x(v))]x(v)\\\notag
               &=\sum\limits_{v\in V}\delta_V(v)[(avg^*\circ avg)(x)(v)-n(v)x(v))]x(v)\\
               &=\sum\limits_{v\in V}\sum\limits_{e\in E_v}\frac{\delta_E(e)}{|e|}[(avg(x))(e)-x(v))]x(v).
\end{align}
The contribution of the hyperedge $e=\{v_1,v_2,\ldots,v_k\}$ in the sum of the \Cref{nsdl} \begin{align}
    \notag &=\sum\limits_{v\in e}\frac{\delta_E(e)}{|e|}[(avg(x))(e)-x(v))]x(v)\\
    \notag &=\sum\limits_{i= 1}^k\frac{\delta_E(e)}{|e|^2}\left[\left\{\sum\limits_{j=1 }^k x(v_j)\right\}-|e|x(v_i))\right]x(v_i)\\
    \notag &=-\frac{1}{2}\frac{\delta_E(e)}{|e|^2}\sum_{i,j=1}^k(x(v_i)-x(v_j))^2\le 0.
\end{align}
Thus the \Cref{nsdl} becomes,
\begin{align}\label{nsdeold}
 (L(x),x)_V=-\sum_{e\in E}   \frac{1}{2}\frac{\delta_E(e)}{|e|^2}\sum_{u,v\in e}(x(u)-x(v))^2\le 0.
\end{align}
Hence the proof follows.
\end{proof}
Note that in \Cref{nsdeold}, the term $(x(u)-x(v))^2$ appear twice in the sum $\sum\limits_{u,v\in e}(x(u)-x(v))^2$, first as $(x(u)-x(v))^2$ and then as $(x(v)-x(u))^2$. Therefore, the \Cref{nsdeold} can also be expressed as 
\begin{align}\label{nsde}
 (L(x),x)_V
 = -\sum_{e\in E}   \frac{\delta_E(e)}{|e|^2}\sum_{\{u,v\}\subset e}(x(u)-x(v))^2.
\end{align}
\begin{prop}\label{evec1}
$0$ is an eigenvalue of $L$ and if the hypergraph is connected then the eigenspace of $0$ is $\langle\mathbf{1}\rangle$, the vector space generated by $\mathbf{1}$.
\end{prop}
\begin{proof}
Since, $L(\mathbf{1})(v)=0$ for all $v\in V$, $0$ is an eigenvalue of $L$ with an eigenvector $\mathbf{1}$.

If $x$ belongs to the eigenspace of $L$ corresponding to the eigenvalue $0$ and the hypergraph is connected, then by \Cref{nsde} we have $x(u)=x(v)$ for all $u,v\in V$. Thus the proof follows.
\end{proof}
By \Cref{nsdo} and \Cref{evec1}, other than $0$ all the eigenvalues of $L$ are negative and for a connected hypergraph, the eigenspace of the eigenvalue $0$ is $\langle\mathbf{1}\rangle$.  Hence, if $x(t)$ is a solution of the differential equation
$$\dot{x}(t)=L(x(t))$$ then as $t\to \infty$, among all the components of decomposed vector $x(t)$ along the eigenvectors of $L$ only the component along $\mathbf{1}$ survives and rest of all tend to $0$. Thus, as $t\to \infty$, any solution of the given differential equation converge to the vector space $\langle\mathbf{1}\rangle$. Therefore, $L$ is a reasonable candidate for being the diffusion operator corresponding to a hypergraph on the space of all real-valued functions on the set of all the vertices, $V$. 

Suppose $|V|=N$. By \Cref{nsdo} and \Cref{evec1}, there exists a collections of non-negative reals $\{\lambda_i(G)\}_{i=1}^N$ (or simply $\{\lambda_i\}_{i=1}^N$ if there is no scope of confusion regarding the hypergraph) such that $-\lambda_i$ is an eigenvalue of $L$ for all $i\in \mathbb{N}_{N}$. Suppose the indices $i(\in \mathbb{N}_{N})$ are chosen in such a way that $\lambda_i\le\lambda_{i+1}$. By \Cref{evec1}, $\lambda_1=0$. The \textit{Rayleigh quotient} $R(-L,x)$ of $-L$ and  nonzero $x(\in \mathbb{R}^V$) is $\frac{(-Lx,x)_V}{(x,x)_V}$. Since $L$ is self-adjoint, we can assume that  $\{\mathbf{1},z_2,\ldots, z_N\}$ is the orthonormal basis of $\mathbb{R}^V$ consisting of the eigenfunctions of $L$ and $z_i(\in \mathbb{R}^V)$ is the eigenfunction of $L$ corresponding to the eigenvalue $\lambda_i$. The Rayleigh quotient reaches its minimum value $\lambda_1=0$  when $x=\mathbf{1}$, the eigenvector of $L$ corresponding to the eigenvalue $\lambda_1=0$. Moreover, $\lambda_2=\inf\limits_{\mathclap{x\in \langle\mathbf{1}\rangle^\perp-\{\mathbf{0}\}}}R(-L,x)$ and the Rayleigh quotient reaches the infimum value at $x=z_2$, the eigenvector of $L$, corresponding to the eigenvalue $\lambda_2$. The multiplicity of $0$ as an eigenvalue of the graph Laplacian is equal to the number of connected components of the graph. This is an well known result for the graphs. One can conclude the same for hypergraph Laplacian. The proof for hypergraph is almost same that works for graphs.

\begin{prop}
Multiplicity of the zero eigenvalue of the diffusion matrix $L$ of a hypergraph $G$ is equal to the number of connected components in $G$.
\end{prop}
\begin{proof}
  Suppose that the multiplicity of the zero eigenvalue is $k$. 
Let $S$ be the eigenspace of  the eigenvalue $0$ of $L$. Let $(V_1,E_1)$ $,(V_2,E_2)$ $,\ldots, (V_k,E_k)$ be the $k$ components of the hypergraph $G$. Evidently, by \cref{nsde}, for all $z\in S$,
\begin{equation}\label{equal}
    0=(Lz,z)_V=-\sum_{e\in E}   \frac{1}{2}\frac{\delta_E(e)}{|e|^2}\sum_{u,v\in e}(z(u)-z(v))^2.
\end{equation}
It is evident from \Cref{equal} that for all $z\in S$, $z$ is constant within each connected component of the hypergraph. Therefore, for all $z\in S$, there exists $z_1,z_2,\ldots,z_k\in \mathbb{R} $ such that $z(u)=z_i$ for all $u\in V_i$ where $i\in\{1,2,\ldots,k\}$. This association of each elements of $S$ to $k$ real numbers motivates us to define the linear map $\mathfrak{g}:S\to \mathbb{R}^k$ by $\mathfrak{g}(z):=(z_1,z_2,\ldots,z_k)$. Using \Cref{equal}, one can easily verify that $\mathfrak{g}$ is an isomorphism. therefore, the geometric multiplicity of $0$, which is the dimension of the eigenspace $S$ is exactly equal to $k$. Since $L$ is a self-adjoint operator, for any eigenvalue of $L$ the algebraic multiplicity is equal to the geometric multiplicity. Therefore, the number of components, $k$ is equal to the multiplicity of $0$ as an eigenvalue of $L$.
\end{proof}
Now we provide some results on the eigenvalues and eigenvectors of the diffusion operator of some classes of hypergraphs. The following results allow us to determine several eigenvalues of the same simply by looking at the hypergraphs. According to the definition of the diffusion operator, a hypergraph corresponds to a class of diffusion operator. More precisely, a hypergraph along with a particular choice of $(\delta_V,\delta_E)$ induces a diffusion operator.Therefore, naturally the eigenvectors and eigenvalues of the diffusion operator depends both on the structures of the hypergraphs and the choices of the inner products. In the following results, we determined the eigenvalues and the eigenvectors of the diffusion operator with two types of specifications- $i)$ the conditions on the structure of the hypergraphs specify the class of hypergraphs, and $ii)$ the conditions on $\delta_E,\delta_V$  specify the subclass of the diffusion operators. 
\begin{thm}\label{cute-new}
If $G=(V,E)$ is a hypergraph such that
\begin{itemize}
    \item [(i)] \label{cute-new-structure}$E_k=\{e_1,e_2,\ldots,e_k\}\subset E$ with $W=\bigcap\limits_{e\in E_k}e$ and $|W|\ge 2$, and $e\cap W=\emptyset$ for all $e\in E\setminus E_k$,
    \item[(ii)] \label{delta-cute-new}$\delta_V(v)=c$ for all $v\in W$, 
    for some fixed $c
    \in\mathbb{R}$,
\end{itemize}
then $-\frac{1}{c}\sum\limits_{e\in E_k}\frac{\delta_E(e)}{|e|}$ is an eigenvalue of the diffusion operator $L$ and the dimension of the corresponding eigenspace is at least $|W|-1$.
\end{thm}
\begin{proof}
Suppose that $W=\{v_0,v_1,\ldots,v_s\}$. Corresponding to each $v_i$, for all $i=1,2,\ldots,s$, we define $y_i\in \mathbb{R}^V$ as, 
$$
y_i(v)=
\begin{cases}
-1&\text{~if~} v=v_0\\
 \phantom{-}1&\text{~if~} v=v_i\\
 \phantom{-}0&\text{~otherwise.~}
\end{cases} 
$$ 
 We enlist below some crucial observations on $y_i$, for all $i=1,2,\ldots,s$.
 \begin{itemize}
     \item[(a)]If $v\in V\setminus W$ then $y_i(v)=0$. 
     \item[(b)]  If $e\notin E_k $ then $e\cap W=\emptyset$ and therefore, $y_i(v)=0$ for all $v\in e$. If $e\in E_k$ then from the definition of $y_i$ we have $\sum\limits_{v\in e}y_i(v)=0$. Therefore, $\sum\limits_{v\in e}y_i(v)=0$ for all $e\in E$.
     \item[(c)] Therefore, $  (Ly_i)(v)=\sum\limits_{e\in E_v}\frac{\delta_E(e)}{\delta_V(v)}\frac{1}{|e|^2}\sum\limits_{u\in e}(y_i(u)-y_i(v))=-\sum\limits_{e\in E_v}\frac{\delta_E(e)}{\delta_V(v)}\frac{1}{|e|}y_i(v)$.
     \item[(d)] Evidently, $E_v=E_k$ for all $ v\in W$, 
 \end{itemize}
 Since $\delta_V(v)=c,
 $ 
 for all $v\in W$, then 
 by the above observations we conclude that for all $i=1,2,\ldots,s$, 
\begin{align*}
  (Ly_i)(v) &=
  \begin{cases}
  -\sum\limits_{e\in E_k}\frac{\delta_E(e)}{c}\frac{1}{|e|}y_i(v) & \text{~if~} v\in W,\\
  0 & \text{~otherwise.~}
  \end{cases}
 \end{align*} 
 Thus $(Ly_i)(v) = -\frac{1}{c}\sum\limits_{e\in E_k}\frac{\delta_E(e)}{|e|}y_i(v).$

Therefore, $-\frac{1}{c}\sum\limits_{e\in E_k}\frac{\delta_E(e)}{|e|} $ is an eigenvalue of $L$ with the eigenvectors $y_1,y_2,\ldots, y_s$, respectively. Since  $$(\sum\limits_{i=1}^sc_iy_i)(v)=
\begin{cases}
-\sum\limits_{i-1}^s c_i&\text{~if~} v=v_0,\\
 \phantom{-}c_i&\text{~if~} v=v_i,\\
 \phantom{-}0&\text{~otherwise,~}
\end{cases}$$ for $c_1,c_2,\ldots,c_s\in \mathbb{R}$,  $ (\sum\limits_{i=1}^sc_iy_i)=0$ implies $c_i=0$ for all $i=1,2,\ldots, s$.  Therefore, $\{y_1,y_2,\ldots, y_s\}$ is linearly independent and the dimension of the eigenspace of the above mentioned eigenvalue is at least $s$.
\end{proof}
We provide below some examples related to the above result.
\begin{nt}
\begin{enumerate}
    \item Let us recall \Cref{ex-diff}(\ref{L1}). If we put $\delta_V(v)=1$ and $\delta_E(e)=|e|^2$, then the diffusion operator $L$ becomes the negative of the Laplacian matrix, described in \cite{rodriguez2003Laplacian,rodriguez2009Laplacian}. In this case, $\delta_V$ is constant function and $\sum\limits_{e\in E_k}\frac{\delta_E(e)}{|e|}=\sum\limits_{e\in E_k}|e|$ is always a constant and thus, the condition (ii) of \Cref{cute-new} holds trivially and not required to be mentioned in this case. That is, in this case, $-\sum\limits_{e\in E_k}|e|$ is an eigenvalue of the diffusion operator and  $\sum\limits_{e\in E_k}|e|$ is an eigenvalue of the Laplacian operator with the multiplicity $|W|-1$.
    \item In \Cref{ex-diff}(\ref{L2}) we have seen, for $\delta_V(v)=1$ and $\delta_E(e)=\frac{|e|^2}{|e|-1}$ the diffusion operator becomes the negative of the Laplacian operator mentioned in \cite{MR4208993}. In this case  $\sum\limits_{e\in E_k}\frac{\delta_E(e)}{|e|}=\sum\limits_{e\in E_k}\frac{|e|}{|e|-1}$ is always a constant and therefore, also for this particular diffusion operator the condition $(ii)$ always holds. Evidently, here the Laplacian eigenvalue is $\sum\limits_{e\in E_k}\frac{|e|}{|e|-1}$ with the multiplicity $|W|-1$.
    \item  Recall \Cref{ex-diff-hy}(\ref{L3}). If $\delta_V(v)=|E_v|$, and $\delta_E(e)=\frac{|e|^2}{|e|-1}$ then the diffusion operator becomes the negative of the normalized Laplacian described in \cite[Equatioin-14]{MR4208993}. Note that in \Cref{cute-new}, $E_v=E_k$ for all $v\in W$ and thus, $\delta_V(v)=|E_k|=k$ for all $v\in W$. Therefore, by \Cref{cute-new}, $\frac{1}{k}\sum\limits_{e\in E_k}\frac{|e|}{|e|-1}$ is an eigenvalue of the normalized Laplacian with the multiplicity $|W|-1$.
\end{enumerate}
\end{nt}
\begin{exm}
    Consider the hypergraph $H(V,E)$ where $V=[20]=\{n\in \mathbb N:n\le 20\}$ and $E=\{e_1=\{1, 2,3,4\},e_2=\{1,2, 5,6,7\},e_3=\{1,2,8,9,10\},e_4=\{1,2,11,12,13,14\},e_5=\{1,2,15,16,17,18,19,20\}\}$. Since $|W|=\bigcap\limits_{i=1}^5e_i=\{1,2\}$, we have the followings.
    \begin{enumerate}
        \item In the framework of \cite{MR4208993}, one eigenvalue of the Laplacian of $H$ is $ \sum\limits_{i=1}^5\frac{|e_i|}{|e_i|-1}=\frac{1297}{210}$ and an eigenvalue of the normalized Laplacian matrix of $H$ is $\frac{1}{5}\sum\limits_{e\in E_k}\frac{|e|}{|e|-1}=\frac{1297}{1050}$.
        \item In the framework of \cite{rodriguez2003Laplacian,rodriguez2009Laplacian,bretto2013hypergraph}, one eigenvalue of the Laplacian of $H$ is $ \sum\limits_{i=1}^5|e_i|=28$.
    \end{enumerate}
    \end{exm}
   
\begin{cor}\label{cute1}
If $G=(V,E)$ is a hypergraph satisfying the following conditions
\begin{itemize}
    \item [(i)] the intersection of all the hyperedges contains at least two vertices, that is, $|\bigcap\limits_{e\in E}e|\ge2$,
    \item [(ii)]the function $\delta_V$ is constant on $\bigcap\limits_{e\in E}e$, that is, there exists $c\in \mathbb{R}^+$ such that $\delta_V(v)=c$ for all $v\in \bigcap\limits_{e\in E}e$,
\end{itemize}
 then $-\sum\limits_{e\in E}\frac{\delta_E(e)}{c}\frac{1}{|e|}$ is an eigenvalue of the diffusion operator $L$ and the dimension of the corresponding eigenspace is at least $|\bigcap\limits_{e\in E}e|-1$.
\end{cor}
\begin{proof}This result directly follows from the \Cref{cute-new}
\end{proof}
\begin{thm}\label{cute2}
Let $G=(V,E)$ be a hypergraph. Suppose that there exists an hyperedge, $e_0\in E$, such that
\begin{enumerate}
    \item[(i)] $e_0=e_u\cup e_v$ where 
    $e_u\cap e_v=\emptyset$,
    \item[(ii)] $|e_u|\ge 2$,
    \item[(iii)]{\label{con-3}} $e\cap e_u=\emptyset$ for all $e(\neq e_0)\in E$.
\end{enumerate}
 If $\delta_V(v)=c$ for all $v\in e_u$ then $-\frac{\delta_E(e_0)}{c}\frac{1}{|e_0|}$ is an eigenvalue of the diffusion operator $L$ of the hypergraph $G$ with multiplicity at least $|e_u|-1$.
\end{thm}
\begin{proof}
Suppose that $ e_u=\{u_0,u_1,\ldots, u_k\}$. For each $u_i $, $i=1,2,\ldots, k$, we define $y_i\in \mathbb{R}^V$ by 
$$
y_i(v)=
\begin{cases}
\phantom{-}1 \text{~if ~} v=u_i,\\
-1\text{~if~} v=u_0, \\
\phantom{-}0 \text{~otherwise.~}
\end{cases}
$$
Now we have the following observations on $y_i$ for $i=1,2,\ldots,k$.
\begin{itemize}
    \item [(a)]
    Evidently, $\sum\limits_{u\in e}y_i(u)=0$ for all $e\in E$ and $i=1,2, \ldots,k$ because if $e\ne e_0$ then $y_i(v)=0$ for all $v\in e$ and $\sum\limits_{u\in e}y_i(u)=1-1=0$. Thus, for all $v\in V$, one has by \Cref{L}
    \begin{align*}
         (Ly_i)(v)
         &=-\sum\limits_{e\in E_v}\frac{\delta_E(e)}{\delta_V(v)}\frac{1}{|e|}y_i(v).
         \end{align*}
    \item[(b)] For all $v\in e_u$, $E_v=\{e_0\}$. 
    Thus, for all $v\in e_u$, we have 
    \begin{align*}
         (Ly_i)(v)&
         =-\frac{\delta_E(e_0)}{\delta_V(v)}\frac{1}{|e_0|}y_i(v).
         \end{align*}
         
\end{itemize} Since $\delta_V(v)=c$ for all $v\in e_u$, by the above observations one has \begin{align}
     (Ly_i)(v)
      &=\begin{cases}
     -\frac{\delta_E(e_0)}{c}\frac{1}{|e_0|}y_i(v) & \text{~if~}v \in e_u,\\
     0&\text{~otherwise~}.
     \end{cases}\notag
\end{align}     
Thus $(Ly_i)(v) = -\frac{\delta_E(e_0)}{c}\frac{1}{|e_0|}y_i(v).$
Therefore, $-\frac{\delta_E(e_0)}{c}\frac{1}{|e_0|}$ is an eigenvalue of $L$ with the eigenvectors $y_1,y_2,\ldots,y_k$, respectively. Note that,  $$\left(\sum\limits_{i=1}^kc_iy_i\right)(v)=
\begin{cases}
-\sum\limits_{i=1}^kc_i & \text{~if~} v=u_0,\\
c_i & \text{~if~}v=v_i,\\
0&\text{~otherwise~.}
\end{cases}$$ Therefore, $\sum\limits_{i=1}^kc_iy_i=0$ leads  to $c_i=0$ for all $i=1,2,\ldots,k$ and $\{y_1,y_2,\ldots, y_k\}$ is a linearly independent subset of the eigenspace of $-\frac{\delta_E(e_0)}{c}\frac{1}{|e_0|} $. This proves that the multiplicity of the eigenvalue $-\frac{\delta_E(e_0)}{c}\frac{1}{|e_0|}$ is at least $k$.
\end{proof}
\begin{nt}\label{petal_lap_nt}
\begin{enumerate}
    \item\label{norm_lap_petal} Using conditions of \Cref{cute2}, one has $E_v=\{e_0\}$ for all $v\in e_u$. Recall \Cref{ex-diff-hy}(\ref{L3}). If $\delta_V(v)=|E_v|$, and $\delta_E(e)=\frac{|e|^2}{|e|-1}$ then the diffusion operator becomes the negative of the normalized Laplacian described in \cite[Equatioin-14]{MR4208993}. Since $\delta_V(v)=|E_v|=1$ for all $v\in e_u$, by \Cref{cute2}, $\frac{|e_0|}{|e_0|-1}$ is an eigenvalue of the normalized Laplacian with the multiplicity $|e_u|-1$. Moreover, according to \cite{MR4208993}, the normalized Laplacian matrix described in \cite[Equation-16]{MR4208993} is similar to that of \cite[Equation-14]{MR4208993}. Therefore, both the matrices have an eigenvalue $\frac{|e_0|}{|e_0|-1}$ with the multiplicity $|e_u|-1$. 
    \item In \Cref{ex-diff}(\ref{L1}), for $\delta_V(v)=1$ and $\delta_E(e)=|e|^2$, the diffusion operator $L$ becomes the negative of the Laplacian matrix, described in \cite{rodriguez2003Laplacian,rodriguez2009Laplacian,bretto2013hypergraph}. In this case, $\delta_V$ is constant function and $\frac{\delta_E(e_0)}{c|e_0|}=|e_0|$ and thus, in this case the eigenvalue of the Laplacian matrix becomes $|e_0|$.
    \item In \Cref{ex-diff}(\ref{L2}), we have seen, for $\delta_V(v)=1$ and $\delta_E(e)=\frac{|e|^2}{|e|-1}$ the diffusion operator becomes the negative of the Laplacian operator mentioned in \cite{MR4208993}. Since, here, $\frac{\delta_E(e_0)}{c|e_0|}=\frac{|e_0|}{|e_0|-1}$, thus, in this case the eigenvalue of the diffusion operator becomes $-\frac{|e_0|}{|e_0|-1}$ and the eigenvalue of the Laplacian matrix is $ \frac{|e_0|}{|e_0|-1}$.
 \end{enumerate}
\end{nt}
\begin{exm}\sloppy
Consider a hypergraph $H(V,E)$ with $V=[11]=\{n\in \mathbb{N}:n\le11\}$ and $E=\{e_1=\{1,2,3,4,5\},e_2=\{4,5,6,7,10,11\},e_3=\{6,7,8,9\},e_4=\{8,9,10,11\}\}$. Since $W=\{1,2,3,\}\subset e_1$ and $W\cap e=\emptyset$ for all $e(\neq e_1)\in E$, we have the followings.
   \begin{enumerate}
   
       \item In the framework of \cite{MR4208993}, an eigenvalue of the Laplacian of $H$ is $ \frac{|e_1|}{|e_1|-1}=\frac{5}{4}$. Moreover, by \Cref{petal_lap_nt}(\ref{norm_lap_petal}), $ \frac{|e_1|}{|e_1|-1}=\frac{5}{4}$ is also an eigenvalue of the normalized Laplacians described in \cite{MR4208993}.
        \item In the framework of \cite{rodriguez2003Laplacian,rodriguez2009Laplacian,bretto2013hypergraph}, one eigenvalue of the Laplacian of $H$ is $ |e_1|=5$.
   \end{enumerate}
\end{exm}
\begin{thm}\label{cute3}
Suppose that $G=(V,E)$ be a hypergraph such that $E_0=\{e_0,e_1,\ldots,e_k\}\subset E$  with
\begin{itemize}
    \item [(i)] $W=
    \bigcap\limits_{i=1}^ke_i\neq \emptyset$, and $W\cap e=\emptyset$ for all $e\in E\setminus E_0$,
    \item[(ii)]for all $i=0,1,\ldots,k$ there exists an $F_i\subset V$ such that $e_i=W\cup F_i$ with $|F_i|=t$ and $F_i\cap W=\emptyset$ for all $i$,
    \item[(iii)] $F_i\cap e=\emptyset$ for all $e(\ne e_i)\in E$.
    \item[(iv)]there exists $c,\omega\in \mathbb{R}$ such that $\delta_V(v)=c$ for all $v\in \bigcup\limits_{i=0}^kF_i$, and $\frac{\delta_E(e)}{|e|^2}=\omega$ for all $e\in E_0$.
\end{itemize}
 Then $-\frac{\omega}{c}|W|$ is an eigenvalue of $L$ with multiplicity at least $|E_0|-1$.
\end{thm}
\begin{proof}
We define $y_i\in \mathbb{R}^V$ for all $i=1,2,\ldots,k$, as
$$
y_i(v)=
\begin{cases}
-1&\text{~if~} v\in F_0\\
\phantom{-}1&\text{~if~}v\in F_i\\
\phantom{-}0&\text{~otherwise.~}
\end{cases}
$$
 Now, we consider the following cases to prove the result.
\begin{itemize}
    \item [(a)] For $v\in F_j$, one has  $E_v=\{e_j\}$ for any $j=0,1,\ldots,k$. Therefore, \Cref{L} becomes $(Ly_i)(v)=\frac{\delta_E(e_j)}{c|e_j|^2}\sum\limits_{u\in e }(y_i(u)-y_i(v))=\frac{\delta_E(e_j)}{c|e_j|^2}\sum\limits_{u\in W }(y_i(u)-y_i(v))=-\frac{\omega}{c}|W|y_i(v)$.
    \item[(b)] For $v\in W $, clearly $E_v=E_0$ and $y_i(v)=0$. Thus, $(Ly_i)(v)=\sum\limits_{e\in E_0}\frac{\delta_E(e)}{\delta_V(v)}\frac{1}{|e|^2}\sum\limits_{u\in e}(y_i(u)-y_i(v))\\=\sum\limits_{j=0}^k\frac{\delta_E(e_j)}{\delta_V(v)}\frac{1}{|e_j|^2}\sum\limits_{u\in e_j}y_i(u)=\frac{\delta_E(e_i)}{\delta_V(v)}\frac{1}{|e_i|^2}t-\frac{\delta_E(e_0)}{\delta_V(v)}\frac{1}{|e_0|^2}t=\frac{\omega}{c}(t-t)=0$.
    \item[(c)] For $v\in V\setminus\left(W\cup\left(\bigcup\limits_{i=0}^kF_i\right)\right)$, one has $\sum\limits_{u\in e}(y_i(u)-y_i(v))=0$ for all $e\in E_v$. Therefore, $(Ly_i)(v)=0$. 
    \end{itemize}
    Therefore, $-\frac{\omega}{c}|W|$ is an eigenvalue of $L$.
    
    Since $$(\sum\limits_{i=1}^kc_iy_i)(v)=
\begin{cases}
-\sum\limits_{i=1}^kc_i &\text{~if~} v\in F_0,\\
c_i &\text{~if~} v\in F_i,\\
0   &\text{~otherwise,}
\end{cases}$$ we have $\sum\limits_{i=1}^kc_iy_i=0$ if and only if $c_i=0$ for all $i=1,2,\ldots,k$. So, $y_1,y_2,\ldots,y_k$ are linearly independent and the dimension of the eigenspace of the eigenvalue $-\frac{\alpha}{c}\frac{s-1}{s^2}$ of L is at least $k$. Therefore, the multiplicity of the eigenvalue $-\frac{\omega}{c}|W|$ is at least $k=|W|-1$.
\end{proof}
\begin{nt}
\begin{enumerate}
     \item In \Cref{ex-diff}(\ref{L1}), for $\delta_V(v)=1$ and $\delta_E(e)=|e|^2$, the diffusion operator $L$ becomes the negative of the Laplacian matrix, described in \cite{rodriguez2003Laplacian,rodriguez2009Laplacian}. In this case, $\delta_V$ is a constant function and $\frac{\delta_E(e_0)}{|e|^2}=1$. Thus  the condition $(iv)$ of \Cref{cute3} holds trivially. Therefore,  $|W| $ becomes an eigenvalue with the multiplicity $|W|-1$.
    \item In \Cref{ex-diff}(\ref{L2}) we have seen, for $\delta_V(v)=1$ and $\delta_E(e)=\frac{|e|^2}{|e|-1}$ the diffusion operator becomes the negative of the Laplacian operator mentioned in \cite{MR4208993}. Here  one can also verify easily that  the condition $(iv)$ of \Cref{cute3} holds if all the hyperedges in $E_0$ are of same cardinality. Therefore, in this case $\frac{1}{|e|-1}|W| $ is an eigenvalue with the multiplicity $|W|-1$.
    \item Note that for all $i=0,1,\ldots,k$, if $v\in F_i$  then $E_v=\{e_i\}$. Therefore, $|E_v|=1$ for all $v\in \bigcup\limits_{i=0}^kF_i$.  Let us recall \Cref{ex-diff-hy}(\ref{L3}). Now if $\delta_V(v)=|E_v|$, and $\delta_E(e)=\frac{|e|^2}{|e|-1}$ then the diffusion operator becomes the negative of the normalized Laplacian described in \cite[Equatioin-14]{MR4208993}. In this framework, $\delta_V(v)=|E_v|=1$ for all $v\in \bigcup\limits_{i=0}^kF_i$ and thus by \Cref{cute3}, we get $\frac{1}{|e|-1}|W| $ is an eigenvalue of the normalized Laplacian matrix with the multiplicity $|W|-1$.
\end{enumerate}
\end{nt}
We provide an application of the above result in \Cref{spectra_ex}.
\subsection{Spectra of Diffusion Operator of Some Specific Hypergraphs }\label{spectra_ex}
Now we recall some definitions of special type of hypergraphs from \cite{MR3116407,andreotti2020spectra} and derive the eigenvalues of their diffusion operators.
\begin{df}[Cored vertex]\cite[Definition 2.3]{MR3116407}
Suppose that $G(V,E)$ is a hypergraph. If for all $e\in E$, there exists $v_e\in e$ such that $v_e\notin e_j$ for all $e_j(\neq e)\in E$ then $G$ is called a cored hypergraph. A vertex with degree one is referred to as a cored vertex, and a vertex with degree greater than one is referred to as an \textit{intersectional} vertex. 
\end{df}
According to \cite{andreotti2020spectra}, if a hyperedge has only one cored vertex then the core vertex is called a \textit{pendant vertex}. Moreover, two vertices $u,v$ of a hypergraph are \textit{twins} if they belong to the exactly same hyperedge(s). 
Note that in \Cref{cute2}, all the elements of $e_u$ are cored vertex and any pair of vertices in $e_u$ are twins.  Moreover, \Cref{cute2} can be applied if there exists at least two cored vertex which are twins. In \Cref{cute3}, each $u_i$ is the only cored vertex in $e_i $, that is, each $u_i$ is a pendant and other than $u_i$, all the vertices in $e_i$ are intersectional. In \Cref{cute1}, the condition (1) can be restated as there exists at least a pair of twin vertices belongs to all the hyperedges. Now we are going to apply our results to determine the eigenvalues of some classes of hypergraphs that has cored vertices, twin vertices, and intersectional vertices.
\subsubsection{Complete Spectra of the Diffusion Operator of Hyperflowers}
\begin{df}[Hyperflowers]\cite{andreotti2020spectra} \label{hyperflower} A $(l,r)$-hyperflower  with $t$-twins is a hypergraph $G=(V,E)$ where $V$ can be expressed as the disjoint partition $V=U\cup W$ with the following listed property.
\begin{itemize}
\item[(a)] The set $U$ can be partitioned into disjoint $t$-element sets as $U=\bigcup\limits_{i=1}^l U_i$. That is $|U_i|=t$, $U_i=\{u_{is}\}_{s=1}^t$ for all $i=1,2\ldots,l$ and $U_i\cap U_j=\emptyset$ for all $i\neq j$ and $i,j=1,\ldots,l$.
\item[(b)] There exists $r$-disjoint set of vertices $e_1,\ldots,e_r\in \mathcal{P}(W)$, the power set of $W$, such that, $W=\bigcup\limits_{j=1}^re_j$ and
$E=\{e_{ki}:e_{ki}=e_k\cup U_i,k=1,2,\ldots,r;i=1,2,\ldots,l\}$.
\end{itemize}
If $v\in U$, then $v$ is called a peripheral vertex.
\end{df}
Suppose that $E_k=\{e_{ki}\}_{i=1}^l$ for any $k=1,2,\ldots ,r$. Evidently, $e_k=\bigcap\limits_{e\in E_k}e$ and $e\cap e_k=\emptyset$ for all $e\in E\setminus E_k$. Therefore, by $\Cref{cute-new}$, if $\delta_V(v)=c_k$ and $\sum\limits_{e\in E_v}\frac{\delta_E(e)}{|e|}=\mu_k$ for all $v\in e_k$ then $-\frac{\mu_k}{c_k}$ is an eigenvalue of the diffusion operator $L_G$ with eigenspace of dimension at least $|e_k|-1$ for all $k$.
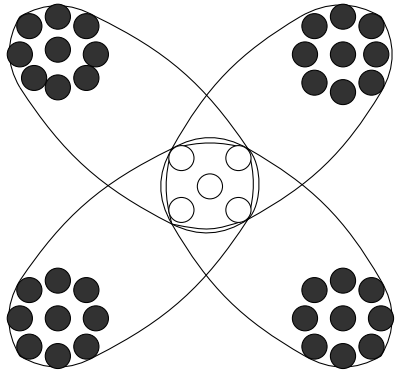
\begin{figure}[ht]
    \centering
\begin{tikzpicture}[scale=0.25]
	\begin{pgfonlayer}{nodelayer}
	\node [style=blackfilled] (0) at (-8, 7.25) {};
		\node [style=blackfilled] (1) at (-6.5, 8.5) {};
		\node [style=blackfilled] (2) at (-9.5, 8.5) {};
		\node [style=blackfilled] (3) at (-6.5, 5.75) {};
		\node [style=blackfilled] (4) at (-9.25, 5.75) {};
		\node [style=blackfilled] (5) at (-10, 7) {};
		\node [style=blackfilled] (6) at (-6, 7) {};
		\node [style=blackfilled] (7) at (-8, 9) {};
		\node [style=blackfilled] (8) at (-8, 5.25) {};
		\node [style=blackfilled] (9) at (7, 7) {};
		\node [style=blackfilled] (10) at (8.5, 8.5) {};
		\node [style=blackfilled] (11) at (5.5, 8.5) {};
		\node [style=blackfilled] (12) at (8.5, 5.5) {};
		\node [style=blackfilled] (13) at (5.5, 5.5) {};
		\node [style=blackfilled] (14) at (5, 7) {};
		\node [style=blackfilled] (15) at (8.75, 7) {};
		\node [style=blackfilled] (16) at (7, 9) {};
		\node [style=blackfilled] (17) at (7, 5) {};
		\node [style=blackfilled] (18) at (7, -7) {};
		\node [style=blackfilled] (19) at (8.5, -5.5) {};
		\node [style=blackfilled] (20) at (5.5, -5.5) {};
		\node [style=blackfilled] (21) at (8.5, -8.5) {};
		\node [style=blackfilled] (22) at (5.5, -8.5) {};
		\node [style=blackfilled] (23) at (5, -7) {};
		\node [style=blackfilled] (24) at (9, -7) {};
		\node [style=blackfilled] (25) at (7, -5) {};
		\node [style=blackfilled] (26) at (7, -9) {};
		\node [style=blackfilled] (27) at (-8, -7) {};
		\node [style=blackfilled] (28) at (-6.5, -5.5) {};
		\node [style=blackfilled] (29) at (-9.5, -5.5) {};
		\node [style=blackfilled] (30) at (-6.5, -8.5) {};
		\node [style=blackfilled] (31) at (-9.5, -8.5) {};
		\node [style=blackfilled] (32) at (-10, -7) {};
		\node [style=blackfilled] (33) at (-6, -7) {};
		\node [style=blackfilled] (34) at (-8, -5) {};
		\node [style=blackfilled] (35) at (-8, -9) {};
		\node [style=black] (36) at (0, 0) {};
		\node [style=black] (37) at (-1.5, 1.5) {};
		\node [style=black] (38) at (1.5, -1.25) {};
		\node [style=black] (39) at (1.5, 1.5) {};
		\node [style=black] (40) at (-1.5, -1.25) {};
		\node [style=none] (41) at (-10, 9) {};
		\node [style=none] (42) at (-6, 9) {};
		\node [style=none] (43) at (-10, 5.25) {};
		\node [style=none] (44) at (-2, -2) {};
		\node [style=none] (45) at (2, 2) {};
		\node [style=none] (46) at (2, -1.75) {};
		\node [style=none] (47) at (9, 9) {};
		\node [style=none] (48) at (-2, 2) {};
		\node [style=none] (49) at (4.75, 9) {};
		\node [style=none] (50) at (9, 5) {};
		\node [style=none] (51) at (-10, -5) {};
		\node [style=none] (52) at (-10, -9) {};
		\node [style=none] (53) at (-6, -9) {};
		\node [style=none] (54) at (5, -9) {};
		\node [style=none] (55) at (9, -9) {};
		\node [style=none] (56) at (9, -5) {};
		\node [style=none] (57) at (-0.25, 2.5) {};
	\end{pgfonlayer}
	\begin{pgfonlayer}{edgelayer}
		\draw [bend right=15] (48.center) to (51.center);
		\draw [bend left=15] (46.center) to (53.center);
		\draw [bend left=15] (48.center) to (45.center);
		\draw [bend left=15] (45.center) to (46.center);
		\draw [bend right] (51.center) to (52.center);
		\draw [bend right] (52.center) to (53.center);
		\draw [bend left] (48.center) to (45.center);
		\draw [bend left=15] (45.center) to (56.center);
		\draw [bend left] (56.center) to (55.center);
		\draw [bend left] (55.center) to (54.center);
		\draw [bend left=15] (54.center) to (44.center);
		\draw [bend left=15] (44.center) to (48.center);
		\draw [bend right] (41.center) to (43.center);
		\draw [bend right=15] (43.center) to (44.center);
		\draw [bend right=15, looseness=1.25] (44.center) to (46.center);
		\draw [bend right=45, looseness=0.75] (46.center) to (45.center);
		\draw [bend right=15] (45.center) to (42.center);
		\draw [bend right] (42.center) to (41.center);
		\draw [bend right] (47.center) to (49.center);
		\draw [bend right=15] (49.center) to (48.center);
		\draw [bend right] (48.center) to (44.center);
		\draw [bend right] (44.center) to (46.center);
		\draw [bend right=15] (46.center) to (50.center);
		\draw [bend right] (50.center) to (47.center);
\end{pgfonlayer}
\end{tikzpicture}

    \caption{$(4,1)$-hyperflower with $9$ twins.}
    \label{fig:hyperflower}
\end{figure}

In our study the hyperflowers with $r=1$ are interesting because, in this case, each peripheral vertex is a cored vertex of the hyperflower. We summarize below some crucial observations on $(l,1)$-hyperflower with $t$ twins  (see \Cref{fig:hyperflower}).  Since $r=1$, one has $W=e_1$ and $\bigcap\limits_{e\in E} e=W$.
\begin{itemize}
    \item[(1)] Note that $e_1=W= \bigcap\limits_{e\in E}e$. Therefore, if $|e_1|\ge 2$ and $\delta_V(v)=c_0 $ for all $v\in e_1$ then by \Cref{cute1}, one can conclude that $-\sum\limits_{e\in E}\frac{\delta_E(e)}{c_0}\frac{1}{|e|}$ is an eigenvalue of the diffusion operator $L$  associated with the $(l,1)$-hyperflower with $t$ twins with the multiplicity at least $|e_1|-1=|W|-1$.
    \item[(2)] Since, for any hyperedge, all the peripheral vertices belong to the hyperedge are cored vertices and any two of them are twins, then by \Cref{cute2}, if $t\ge 2$ and $\delta_V(v)=c_i$ for all $v\in U_i$ then each hyperedge $e_{1i}$ of the hyperflower corresponds to an eigenvalue $-\frac{\delta_E(e_{1i})}{c_i}\frac{1}{|e_{1i}|}$ of $L$ with multiplicity at least $t-1$.
    \item[(3)] Note that if $t\ge 2$ then the total number of vertices in $(l,1)$-hyperflower with $t$ twins is $|V|=l.t+|e_1|$. Evidently, there exists $ l.t+|e_1|$ eigenvectors of $L$, out of which, $(|e_1|-1)+l(t-1)=(l.t+|e_1|)-(l+1)$ can be calculated by using \Cref{cute1} and \Cref{cute2}. We, know among the remaining $l+1$ eigenvectors, $\mathbf{1}$ is an eigenvector of $L$.
    \item[(4)] Suppose that $\delta_V(v)=c$ for all the peripheral vertices $v\in U$ and $ \frac{\delta_E(e)}{|e|^2}=\omega$ for all the hyperedges $e\in E$ then  \Cref{cute3} suggests the existence of the eigenvalue $(-\frac{\omega}{c}|e_1|=)-\frac{\omega}{c}|W|$ with the multiplicity at least $l-1$. 
    \item[(5)] We can easily verify that the last remaining eigenvalue is $-\frac{\omega}{c}(lt+|W|)=-\frac{\omega}{c}|V|$ with an eigenvector $y\in \mathbb{R}^V$ defined by 
    $$
    y(v)=
    \begin{cases}
    \phantom{-}lt &\text{~if~} v\in W\\
    -|W| &\text{~if~} v\in U.
    \end{cases}
    $$.
    \item[(6)] 
    Note that all the eigenvalues of $(l,1)-$ hyperflower with $t$-twins are the multiples of  $\frac{\omega}{c}$. Therefore, if $\frac{\omega}{c}$ is an integer then all the eigenvalues of $L$ are integer.
 \end{itemize}
 In a $(l,1)$-hyperflower, when the central set $W$ is a singleton set $\{v_0\}$ then it becomes a sunflower.
 \begin{df}[Sunflower]\cite[Definition 2.4]{MR3116407} Let $G=(V,E)$ be a $k$-uniform hypergraph. If there exists a disjoint partition of the vertex set $V$ as $V=V_0\cup V_1\cup\ldots \cup V_s$, such that, 
\begin{enumerate}
    \item[(a)] $V_0=\{v_0\}$ and $|V_i|=k-1$ for all $i=1,2,\ldots, s$, 
    \item[(b)]$E=\{e_i=V_0\cup V_i:i=1,2,\ldots, s\}$,
\end{enumerate}
then $G$ is called a $k$-uniform sunflower. Each hyperedge of sunflower is called leaf. The vertex $v_0$ is  referred as the heart of the sunflower. Note that, the degree of the heart is the cardinality of the set $E$, which is also called the size of the sunflower. 
\end{df}
Since the sunflower is a special case of the $(l,1)$-hyperflower, from the list of the eigenvalues of $(l,1)$-hyperflower one can determine the eigenvalues of the diffusion operator associated with sunflower. 
\subsubsection{Spectra of Diffusion Operators of Some More Hypergraphs}
 \begin{df}[Loose Path]\label{path} A $k$-uniform hypergraph $G=(V,E)$ is said to be a $k$-uniform loose path of size $d$ if all the $d$ hyperedges form a sequence $\{e_i \}^d_{i=1}$, such that,
 $e_i\cap e_j=\emptyset$ if $|i-j|>1$ and $|e_i\cap e_j|=1$ if $|i-j|=1$.
 \end{df}Note that if $i\neq 1,d$ then each $e_i$ contains $k-2$ cored vertices. Therefore, if $\delta_V(v)=c_i$ for all cored vertices $v\in e_i$, then by \Cref{cute2}, $-\frac{\delta_E(e_i)}{c_i}\frac{1}{k}$ is an eigenvlue of $L$ with the multiplicity $k-3$. If $k=1,d$ then $e_i$ contains $k-1$ cored vertices, then the multiplicity is $k-2$. If for $k\ge 3$, a $k$-uniform hypergraph $G=(V,E)$ is such that it satisfies  the conditions stated in \Cref{path} with just one exception, which is $|e_1\cap e_d|=2$, then $G$ is called a $k$-uniform \textit{loose cycle} of size $d$. Using \Cref{cute2} we can find the eigenvalues of $L$ of a  loose cycle as we have done it for a loose path.

\section{Spectral bounds for some hypergraph property}\label{bounds}
\begin{df}\label{gendegree}
Let $r\in \mathbb{R}^V$ be defined by $r(v):=\sum\limits_{e\in E_v}\frac{\delta_E(e)}{\delta_V(v)}\frac{|e|-1}{|e|^2}$ and $r_0:=\max\limits_{v\in V}r(v)$. 
\end{df}
 Thus if $\delta_E(e)=w(e)\frac{|e|^2}{|e|-1},\delta_V(v)=1$ then $r(v)=d(v)$ and $r_0=d_{\max}$, where $d_{\max}:=\max\limits_{v\in V}d(v)$.
For unweighted graph, $w(e)=1$ for all $e\in E$ and hence the degree $d(v)$ of a vertex $v$ is the number of  hyperedges contain the vertex $v$. Let $E_v := \{e\in E: v\in e \}$. Thus $d(v)=|E_v|$.
 Clearly for an $m$-uniform hypergraph, $r(v)=\frac{m-1}{|m|^2}\sum\limits_{e\in E_V}\frac{\delta_E(e)}{\delta_V(v)}$.

 Now we compute a bound for $r(v)$ in terms of the spectra of the diffusion operator. Since for  particular choices of inner products, $r(v)$ becomes the degree of $V$, this bound gives a spectral  bound for the vertex degree. We use the techniques described in \cite[3.5., p. 300]{MR318007} to prove the next result.
 \begin{thm} Suppose that $\lambda_2$ and $\lambda_N$ are  the second least and largest eigenvalue, respectively, of $-L$ 
 then
 $$\delta_V(\min)\lambda_2\le \frac{|V|}{|V|-1}\min\limits_{v\in V}r(v)\delta_V(v)\le \frac{|V|}{|V|-1}\max\limits_{v\in V}r(v)\delta_V(v)\le \lambda_N\delta_V(\max).$$
 \end{thm}
 \begin{proof}

 It is easy to show $\Tilde{L}=-L-\lambda_2(I_{|V|}-\frac{1}{|V|}J)$ is positive semidefinite, where $J$ is a square matrix of order $|V|$ with all its entry equal to $1$. 
 Let, for all $v\in V$, $\chi_v\in\mathbb{R}^V$ be defined by $\chi_v(u)=1$ if $u=v$ and otherwise $\chi_v(u)=0$. Hence $(\Tilde{L}\chi_v,\chi_v)_V\ge0$ for all $v\in V$ and so $\lambda_2\le \frac{|V|}{|V|-1}\frac{1}{\delta_V(\min)}\min\limits_{v\in V}\sum\limits_{e\in E_V}\frac{\delta_E(e)}{|e|^2}(|e|-1)$.
 Similarly for the positive semidefinite matrix  $\Tilde{M}=\lambda_N(I-\frac{1}{n}J)_V-(-L)$, we have
  $(\Tilde{M}\chi_v,\chi_v)_V\ge0$ and which implies
 $\frac{|V|}{|V|-1}\max\limits_{v\in V}\sum\limits_{e\in E_V}\frac{\delta_E(e)}{|e|^2}(|e|-1)\le \lambda_N\delta_V(\max)$. This completes the proof.
 \end{proof}
 The maximum and minimum cardinality of hyperedges in a hypergraph are called \textit{rank} $(rk(G))$ and \textit{corank} $(cr(G))$, respectively, of the hypergraph $G$.
 
Removal of a vertex $v\in V$ from each hyperedge containing it, is called \textit{ weak deletion of $v$}. If weak deletion of a set of vertices increase the number of connected components of the hypergraph $G$ then the set is called \textit{weak vertex cut} of the hypergraph $G$. The \textit{weak connectivity number} $\kappa_w(G)$(  or simply $\kappa_w$) is the minimum size of the weak vertex cut in the hypergraph $G$.
\begin{thm}\label{th-upperbound}
Let $G=(V,E)$ be a hypergraph with $|V|(\ge3)$, such that $G$ contains at least one pair of non-adjacent vertices then there exists a constant $\bar{k}$ such that $\lambda_2\le\bar{k} d_{\max}\kappa_W(G)$.
\end{thm}
\begin{proof}
Let $W$ be the the weak vertex cut with $|W|=\kappa_w$. Clearly there exists a partition $V=V_1\cup W\cup V_2$ of $V$ such that no vertex of $V_1$ is adjacent to any vertex in $V_2$. Let us consider $y\in \mathbb{R}^V$ defined by $y(v)=|V_2|$ if $v\in V_1$, $y(v)=-|V_1|$ if $v\in V_2$, and $y(v)=0$ if $v\in W$. 
We define a function $k:V\to \mathbb{R}$ defined by 
$$k(v)=\sum_{e\in E_V}\sum_{u\in e\cap W}\frac{\delta_E(e)}{\delta_V(v)}\frac{1}{|e|^2}.$$
Suppose $\Bar{k}=\sup\limits_{e\in E,v\in V}\left\{\frac{\delta_E(e)}{\delta_V(v)}\frac{1}{|e|^2}\right\}$. It is easy to verify that, for all $v\in V_1\cup V_2$, $k
(v)\le d_{\max}|W|\Bar{k}$.  

Clearly, for any $v\in V_1$, there exists no $e\in E_v$ such that $e\cap V_2\neq\phi$. Hence for any $v\in V_1$, one has

\begin{align*}
    (Ly)(v)&=-\sum_{e\in E_V}\frac{\delta_E(e)}{\delta_V(v)}\frac{1}{|e|^2}\sum_{u\in e\cap w}|V_2|\\
    &=-k(v)|V_2|
\end{align*}
Similarly, for any $v\in V_2$, $(Ly)(v)=k(v)|V_1|$. Hence, \sloppy $(-Ly,y)_V\le\sum_{v\in V}\delta_V(v)k(v)(y(v))^2\le \Bar{k}d_{\max}|W|(y,y)_V$. Thus  $\lambda_2\le \Bar{k}(d_{\max}-1)|W|=\Bar{k}d_{\max}\kappa_W(G)$.
\end{proof}
\begin{rem}In the above result if we put $\delta_E(e)=\frac{|e|^2}{|e|-1}$ for all $e\in E$ and $\delta_V(v)=1$ for all $v\in V$, the diffusion operator $L$ becomes the diffusion operator described in \cite{banerjee2020synchronization}. This operator is also the negative of the Laplacian matrix for hypergraph described in \cite{MR4208993}. With the above choice of $\delta_E$ and $\delta_V$ we have $\Bar{k}=\frac{1}{cr(G)-1}$.
In the above theorem instead of the supremum $\bar{k}$, any upperbound of the set $\left\{\frac{\delta_E(e)}{\delta_V(v)}\frac{1}{|e|^2}\right\}_{e\in E,v\in V}$ yields an upperbound of $\lambda_2$, involving the weak connectivity number. Although we decided to go with the supremum to make the upperbound as sharp as possible.
\end{rem}
\begin{cor}\label{cor-upper}
Let $G=(V,E)$ be a hypergraph with $|V|(\ge3)$ such that $G$ contains at least one pair of non-adjacent vertices and $d_{\max}< cr(G)$ and  $\delta_V(v)=1$ for all $v\in V$, and $\delta_E(e)=w(e)\frac{|e|^2}{|e|-1}$ for all $e\in E $. Then $\lambda_2\le \kappa_W(G)$.
\end{cor}
\begin{proof}
It is easy to verify $ \sum_{e\in E_V}\frac{\delta_E(e)}{\delta_V(v)}\frac{1}{|e|^2}=\sum_{e\in E_V}\frac{w(e)}{|e|}\le\frac{d_{\max}}{cr(G)-1}$. Therefore,
$\Bar{k}\le \frac{d_{\max}}{cr(G)-1} $. Thus, the condition $d_{\max}< cr(G)$ leads us to $\Bar{k}\le 1$. Hence the result follows.
\end{proof}
\Cref{cor-upper} is stated and proved in \cite[Theorem 3.1]{MR4208993} independently. 
For any $S\subset V$, the collection of all the hyperedges  contain vertices from both the sets $S$ and $V\setminus S$ are called the \textit{ edge boundary of the set} $S$. The edge boundary of $S$ is denoted by $\partial S$.
\begin{thm}\label{2n}
Let $G$ be a hypergrph. For any nonempty $S\subset V$, we have
$$4\lambda_2\frac{\delta_V(min)}{\delta_E(\max)}\le \frac{|\partial S| |V|}{|S|(|V|-|S|)}\le \lambda_N\delta_V(\max)\max_{e\in E}\left\{\frac{|e|^2}{(|e|-1)\delta_E(e)}\right\}.
$$
\end{thm}
\begin{proof}
 We consider a function $z_s\in \mathbb{R}^V$ defined by 
 $z_s(v):=\begin{cases}
 |V|-|S| & \text{ if } v\in S
 \\
 -|S|    &\text{ otherwise }
 \end{cases}$ corresponding to the set $S\ne\phi$. Thus by using $A.M\ge G.M$ inequality we have
 \begin{align*}
     (-Lz_s,z_s)_V&=\frac{1}{2}\sum_{e\in E}\frac{\delta_E(e)}{|e|^2}\sum_{u,v\in e}(z_s(u)-z_s(v))^2\\
     &=\frac{1}{2}\sum_{e\in \partial S}\frac{\delta_E(e)}{|e|^2}2|e\setminus S||S\cap e||V|^2\\
     &\le\frac{1}{4}\delta_E(\max)|\partial S||V|^2.
 \end{align*}
\sloppy Now, as $(z_s,z_s)_V\ge \delta_V(\min)|S|(|V|-|S|)|V|$,   we have  $\lambda_2\le\frac{(-Lz_s,z_s)_V}{(z_s,z_s)_V}\le\frac{1}{4}\frac{\delta_E(\max)}{\delta_V(\min)}\frac{|\partial S| |V|}{|S|(|V|-|S|)}$.
 
 As for all $e\in \partial S$,and  $|e-S||S\cap e|\ge |e|-1$   we have  $(-Lz_s,z_s)_V =\frac{1}{2}\sum_{e\in \partial S}\frac{\delta_E(e)}{|e|^2}2|e-S||S\cap e||V|^2\ge |\partial S|\min\limits_{e\in E}\left\{\frac{\delta_E(|e|)}{|e|^2}(|e|-1)\right\}|V|^2$. Thus, $ \frac{|\partial S| |V|}{|S|(|V|-|S|)}\le \lambda_N\delta_V(\max)\max\limits_{e\in E}\left\{\frac{|e|^2}{(|e|-1)\delta_E(e)}\right\}$. This completes the proof.
\end{proof}
\begin{rem}
If $\delta_V(v)=1$ for all $v\in V$ and $\delta_E(e)=\frac{|e|^2}{|e|-1}$ then \Cref{2n} leads us to
$4\lambda_2\frac{cr(G)-1}{rk(G)^2}\le \frac{|\partial S||V|}{|S|(|V|-|S|)}\le \lambda_N$ and this implies the result given in \cite[Theorem 3.2, p-12]{MR4208993}.
\end{rem}
\begin{cor}
Let $G$ be a hypergrph. For any nonempty $S\subset V$, we have
$$4\lambda_2\frac{\delta_V(min)}{\delta_E(\max)}\le \frac{|\delta S| |V|}{|S|(|V|-|S|)}\le \lambda_N\frac{\delta_V(\max)}{\delta_E(\min)}\frac{(rk(G))^2}{cr(G)-1}.
$$
\end{cor}
Let $mc(G):=\max\{|\partial S|:\phi\neq S\subset V\}$ and $bw(G):=\min\left\{|\partial S|:s\subset V,|S|=\left\lfloor\frac{|V|}{2}\right\rfloor\right\}$ be the \textit{maximum cut} and  \textit{bipartition width}, respectively, of a hypergraph $G(V,E)$. Now we have the following corollaries.

\begin{cor}
For any hypergraph $G(V,E)$, $$mc(G)\le\frac{|V|}{4}\lambda_N\delta_V(\max)\max\limits_{e\in E}\left\{\frac{|e|^2}{(|e|-1)\delta_E(e)}\right\}.$$
\end{cor}
\begin{cor}
For any hypergraph $G(V,E)$ if $\alpha(|V|)=\frac{4}{|V|}$ when $|V|$ is even, and $\alpha(|V|)=\frac{4|V|}{|V|^2-1}$ when $|V|$ is odd then $$4\lambda_2\frac{\delta_V(min)}{\delta_E(\max)}\le \alpha(|V|)bw(G)\le \lambda_N\delta_V(\max)\max_{e\in E}\left\{\frac{|e|^2}{(|e|-1)\delta_E(e)}\right\}.
$$
\end{cor}
Now we recall  \textit{Cheeger constant} 
$$h(G):=\min_{S(\ne\phi)\subset V}\left\{\frac{|\partial S|}{\min(|S|,(|V-S|))}\right\}$$ of a hypergraph $G$ \cite{MR4208993}.
\begin{cor}
If $G$ is a connected hypergraph then $\lambda_2\le \frac{1}{2}\frac{\delta_E(\max)}{\delta_V(\min)}h(G)$.
\end{cor}
\begin{proof}
This result follows immediately from \Cref{2n}. Clearly, there exists $S\subset V$ such that $h(G)=\frac{|\partial S|}{|S|}$ and $|S|\le \frac{1}{2}|V|$.  This leads us to $\frac{|V|}{|V|-|S|}\le 2$. Hence by \Cref{2n}, $\lambda_2\le \frac{1}{2}\frac{\delta_E(\max)}{\delta_V(\min)}h(G)$.  
\end{proof}
We recall that in an $m$- uniform hypergraph $G(V,E)$,  $|e|=m$ for all $e\in E$. For any $m$-uniform hypergraph, the \Cref{2n} can be represented as follows.
\begin{prop}\label{rod-gen} Let $G$ be an $m$-uniform hypergraph. Suppose $\gamma:\mathbb{N}\to\mathbb{R}$ is defined by
$\gamma(m)=1$ if $m$ is even and $\gamma(m)=\frac{m^2}{m^2-1}$ if $m$ is odd then for any nonempty $S\subset V$, we have
$$\lambda_N\delta_V(\max)\max_{e\in E}\left\{\frac{m^2}{(m-1)\delta_E(e)}\right\}\ge \frac{|\partial S| |V|}{|S|(|V|-|S|)}\ge4\lambda_2\frac{\delta_V(min)}{\delta_E(\max)}\gamma(m) .
$$
\end{prop}
\begin{proof} The proof is similar as described in
\Cref{2n}, with the fact that for any $m$-hyperedge with $e\cap S=s$, we have $(m-s)s\le \begin{cases}
                    \frac{m^2}{4}& \text{~if~}  $m$ \text{~is even~}  \\
                    \frac{m^2-1}{4} & \text{~if~} $m$ \text{~is odd~}
\end{cases}$.  
\end{proof}
\begin{rem}
If $\delta_V(v)=1$ for all $v\in V$ and  $\delta_E(e)=|e|$ then \Cref{rod-gen} provides the same result stated in \cite[Lemma-1, sec.2, p.917]{rodriguez2009Laplacian}.
\end{rem}

Our next result is a generalization of \cite[Theorem-4.2]{MR4208993}. The proof of the same is also similar to the proof of \cite[Theorem-4.2]{MR4208993}. 

\begin{thm}
If $G$ be a hypergraph with $\lambda_2\le r(v)$ for all $v\in V$ then $$\sqrt{\lambda_2(2r_0-\lambda_2)}\ge \frac{\delta_E(\min)}{\delta_V(\max)(rk(G))^2}h(G).$$
\end{thm}
\begin{proof}
Let $z_2$ be the eigen function of $-L$ corresponding to the eigenvalue $\lambda_2$. Suppose that $V_1=\{v\in V:z_2(v)\ge0\}$ and $V_2=\{v\in V:z_2(v)<0\}$.
Let $y\in\mathbb{R}^V$ be defined by $y(v)=z_2(v)$ if $v\in V_1$, otherwise $y(v)=0 $.  

Since, $$\lambda_2(y,y)_V\ge \left(\sum_{e\in E}\frac{\delta_E(e)}{|e|^2}\sum\limits_{\{u,v\}\subset e\cap V_1}(y(u)-y(v))^2\right)-\sum\limits_{e\in \partial V_1}\frac{\delta_E(e)}{|e|^2}\sum\limits_{u\in e\cap V_2;v\in e\cap V_1}z_2(u)z_2(v)$$
and 
$$(2r_0-\lambda_2)(y,y)_V\ge \left(\sum_{e\in E}\frac{\delta_E(e)}{|e|^2}\sum\limits_{\{u,v\}\subset  e\cap V_1}(y(u)+y(v))^2\right)+\sum\limits_{e\in \partial V_1}\frac{\delta_E(e)}{|e|^2}\sum\limits_{u\in e\cap V_2;v\in e\cap V_1}z_2(u)z_2(v), $$ we  conclude that 
\begin{align}\label{l1}
    \lambda_2(2r_0-\lambda_2)(y,y)_V^2&\ge \left(\sum_{e\in E}\frac{\delta_E(e)}{|e|^2}\sum\limits_{\{u,v\}\subset  e\cap V_1}(y(u)-y(v))^2\right)\left(\sum_{e\in E}\frac{\delta_E(e)}{|e|^2}\sum\limits_{\{u,v\}\subset  e\cap V_1}(y(u)+y(v))^2\right)\notag \\&-\alpha\left(4\sum_{e\in E}\frac{\delta_E(e)}{|e|^2}\sum\limits_{\{u,v\}\subset  e\cap V_1}y(u)y(v)+\alpha\right),
\end{align}
where $\alpha=\sum\limits_{e\in \partial V_1}\frac{\delta_E(e)}{|e|^2}\sum\limits_{u\in e\cap V_2;v\in e\cap V_1}z_2(u)z_2(v)$. Clearly $\alpha\le 0$. Since $\lambda_2\le r(v)$, we have $$\sum_{e\in \partial V_1}\frac{\delta_E(e)}{|e|^2}\sum\limits_{u\in e\cap V_2;v\in e\cap V_1}z_2(u)z_2(v)=\sum\limits_{v\in V_1}(r(v)-\lambda_2)\delta_v(v)\ge0.$$ 
Therefore,
\begin{align}\label{l2}
 &\left(4\sum_{e\in E}\frac{\delta_E(e)}{|e|^2}\sum\limits_{\{u,v\}\subset  e\cap V_1}y(u)y(v)+\alpha\right)  
 \notag\\&=\left(2\sum_{e\in E}\frac{\delta_E(e)}{|e|^2}\sum\limits_{\{u,v\}\subset  e\cap V_1}y(u)y(v)+\sum_{e\in \partial V_1}\frac{\delta_E(e)}{|e|^2}\sum\limits_{u\in e\cap V_2;v\in e\cap V_1}z_2(u)z_2(v)\right)\ge 0.
\end{align}
Since $r(v)\ge \lambda_2$, \Cref{l1} and \Cref{l2} imply that \sloppy
\begin{align}\label{c1}
    &\lambda_2(2r_0-\lambda_2)(y,y)_V^2\notag\\&\ge\left(\frac{1}{2}\sum_{e\in E}\frac{\delta_E(e)}{|e|^2}\sum\limits_{u,v\in e}(y(u)-y(v))^2\right)\left(\frac{1}{2}\sum_{e\in E}\frac{\delta_E(e)}{|e|^2}\sum\limits_{u,v\in e}(y(u)+y(v))^2\right).
\end{align}
Now by Cauchy–Schwarz inequality we have
\begin{align}\label{c2}
    &\left(\frac{1}{2}\sum_{e\in E}\frac{\delta_E(e)}{|e|^2}\sum\limits_{u,v\in e}(y(u)-y(v))^2\right)\left(\frac{1}{2}\sum_{e\in E}\frac{\delta_E(e)}{|e|^2}\sum\limits_{u,v\in e}(y(u)+y(v))^2\right)\notag\\&\ge \left(\frac{1}{2}\sum_{e\in E}\frac{\delta_E(e)}{|e|^2}\sum_{u,v\in e}(y^2(u)-y^2(v))\right)^2.
\end{align}
Suppose $t_0<t_1<\ldots<t_k$ are all possible distinct values of $y$ and $V_i=\left\{v\in V:y(v)\ge t_i\right\}$. Clearly $V=V_0\supseteq V_1\supseteq \ldots\supseteq V_k$.
It can be easily verified that \sloppy$\left(\frac{1}{2}\sum_{e\in E}\frac{\delta_E(e)}{|e|^2}\sum_{u,v\in e}(y^2(u)-y^2(v))\right)\ge \frac{\delta_E(\min)}{(rk(G))^2}\sum\limits_{i\in \mathbb{N}_k} |\partial V_i|(t_i^2-t_{i-1}^2)\ge \frac{\delta_E(\min)}{(rk(G))^2}h(G)\sum\limits_{i\in \mathbb{N}_k} | V_i|(t_i^2-t_{i-1}^2)\ge\frac{\delta_E(\min)}{(rk(G))^2}h(G)\sum\limits_{i\in\mathbb{N}_{k}}( | V_i|-|V_{i+1}|)t_i^2 \ge\frac{\delta_E(\min)}{\delta_V(\max)(rk(G))^2}h(G)(y,y)_V$.

Hence \Cref{c1} and \Cref{c2} lead us to $\sqrt{\lambda_2(2r_0-\lambda_2)}\ge \frac{\delta_E(\min)}{\delta_V(\max)(rk(G))^2}h(G)$.
\end{proof}
\begin{thm}
For any hypergraph $G(V,E)$, $4\lambda_N\frac{\delta_V(\max)}{\delta_E(\min)}\ge h(G)$.
\end{thm}
\begin{proof}
Suppose $S\subset V$ be such that $h(G)=\frac{|\partial S|}{|S|}$ and $z_N\in\mathbb{R}^V$ be the eigenvector of $L$ corresponding to the largest eigenvalue $\lambda_N$ . If $\chi_S\in \mathbb{R}^V$ be the characteristic function of the set $S$ then $\lambda_N=\frac{(-Lz_N,z_N)_V}{(z_N,z_N)_V}\ge \frac{(-L\chi_S,\chi_S)_V}{(\chi_S,\chi_S)_V}\ge \frac{1}{4}\frac{\delta_E(\min)}{\delta_V(\max)}h(G)$.    
\end{proof}
\section{General Adjacency operator of a hypergraph }\label{adjacency}
 For graphs, Adjacency matrix can be expressed as the difference of the degree matrix and the Laplacian matrix associated with the graph. Here we  define the general adjacency operator $A_G:\mathbb{R}^V\to \mathbb{R}^V$ for a hypergraph $G(V,E)$ as follows
\begin{equation}\label{genadjacency}
((A_G)x)(v):=(L_G(x))(v)+r(v)x(v),
\end{equation}
for all $x\in\mathbb{R}^V$, where $L_G$ is the diffusion operator of $G$ and $r(\in \mathbb{R}^V)$ is defined in the \Cref{gendegree}. By \Cref{gendegree},  for all $x\in\mathbb{R}^V$,
\begin{align*}
    (L_G(x))(v)
    &=\sum\limits_{e\in E_v}\frac{\delta_E(e)}{\delta_V(v)}\frac{1}{|e|^2}\sum\limits_{u\in e}(x(u)-x(v))\\
    &=\sum\limits_{e\in E_v}\frac{\delta_E(e)}{\delta_V(v)}\frac{1}{|e|^2}\left(\sum\limits_{u\in e;u\neq v}x(u)-\sum\limits_{u\in e;u\neq v}x(v)\right)\\
    &=\sum\limits_{e\in E_v}\frac{\delta_E(e)}{\delta_V(v)}\frac{1}{|e|^2}\sum\limits_{u\in e;u\neq v}x(u)-\sum\limits_{e\in E_v}\frac{\delta_E(e)}{\delta_V(v)}\frac{|e|-1}{|e|^2}x(v)\\
    &=\sum\limits_{e\in E_v}\frac{\delta_E(e)}{\delta_V(v)}\frac{1}{|e|^2}\sum\limits_{u\in e;u\neq v}x(u)-r(v)x(v)
\end{align*}
Therefore, by \Cref{L} we have
\begin{equation}\label{A}
    (A_Gx)(v)=\sum\limits_{e\in E_v}\frac{\delta_E(e)}{\delta_V(v)}\frac{1}{|e|^2}\sum\limits_{u\in e;u\neq v}x(u).
\end{equation}
   Henceforth we simply use $A$ to denote the general adjacency operator of a hypergraph instead of $A_G$ (if there is no confusion about the hypergraph $G$).
   
    Now we compute some eigenvalues and their eigenspaces of the general adjacency operators associated with some classes of hypergraphs. 
   \begin{rem}
   For some specific values of $\delta_E$ and $\delta_V$, the diffusion operator coincides with some conventional operators. Similarly, if we choose $\delta_V(v)=1$ for all $v\in V$ and $\delta_E(e)=\frac{|E|^2}{|e|-1}$ for all $e\in E$,  our adjacency operator becomes the adjacency operator described in \cite{MR4208993}. So all the theorems on adjacency operator stated in this subsection are also valid for the adjacency operator in \cite{MR4208993}.
   \end{rem}
  \begin{thm}\label{adj_cute_1}
   Suppose that $G=(V,E)$ be a hypergraph. If $e_0\in E$ such that
   \begin{itemize}
       \item[(i)] $e_0=e_u\cup e_v$, with $e_u\cap e_v=\emptyset$, and  $|e_u|\ge 2$,
       \item[(ii)]\label{einte0}$e\cap e_u=\emptyset$ for all all $e(\neq e_0)\in E$,
       \item[(iii)] $\delta_V(v)=c$ for all $v\in e_u$,
   \end{itemize}
   then $-\frac{\delta_E(e_0)}{c|e_0|^2}$ is an eigenvalue of $A$ with multiplicity $|e_0|-1$.
  \end{thm}
   \begin{proof}
   Suppose that $e_u=\{u_0,u_1,\ldots,u_k\}$. Corresponding to each $u_i$, for $i=1,\ldots,k$, we define $y_i\in \mathbb{R}^V$ as 
   \[
   y_i(v)=
   \begin{cases}
   1&\text{~if~} v=u_0,\\
   -1&\text{~if~}v=u_i,\\
   0 &\text{~otherwise.~}
   \end{cases}
   \]
   So $(Ay_i)(v)=\sum\limits_{e\in E_v}\frac{\delta_E(e)}{\delta_V(v)}\frac{1}{|e|^2}\sum\limits_{u\in e;u\neq v}y_i(u)$. Now we have the following observations.
   \begin{itemize}
       \item[(a)] For $v=u_0$, one has $E_v=\{e_0\}$. Therefore, $(Ay_i)(u_0)=\frac{\delta_E(e_0)}{\delta_V(u_0)}\frac{1}{|e_0|^2}\sum\limits_{u\in e;u\neq u_0}y_i(u)$. Since $y_i(v)=0$ for all $v\in V$ with $v\neq u_0$ and $v\neq u_i$, evidently, $(Ay_i)(u_0)=\frac{\delta_E(e_0)}{\delta_V(u_0)}\frac{1}{|e_0|^2}y_i(u_i)=-\frac{\delta_E(e_0)}{c}\frac{1}{|e_0|^2}y_i(u_0)$. 
       \item[(b)] Similarly, for any $i=1,2\ldots,k$, we have $(Ay_i)(u_i)=\frac{\delta_E(e_0)}{\delta_V(u_0)}\frac{1}{|e_0|^2}y_i(u_0)=-\frac{\delta_E(e_0)}{c}\frac{1}{|e_0|^2}y_i(u_i)$.
       \item[(c)] For all $j\neq i $ and $j\neq0$,  evidently, $\sum\limits_{u\in e;u\neq u_j}y_i(u)=0$ for all $e\in E_{u_j}$. Therefore, $(Ay_i)(u_j)=0$ for $j\neq 0$ and $j\neq i$.
       \item[(d)] For all $v\notin e_u$, 
       $\sum\limits_{u\in e;u\neq v}y_i(u)=0$ for all $e\in E$ and therefore, $(Ay_i)(v)=0$.
   \end{itemize}
   Therefore, $Ay_i=-\frac{\delta_E(e_0)}{c}\frac{1}{|e_0|^2}y_i$ for all $i=1,2,\ldots ,k$. Evidently, $\{y_i\}_{i=1}^k$ is a linearly independent set.
   The rest of the proof is similar with the proof of
  \Cref{cute2}. 
   \end{proof}

   \begin{thm}\label{adj_cute_2}
   Suppose that $G=(V,E)$ is a hypergraph. If
   \begin{itemize}
       \item[(i)]\label{con1} there exists $E_0=\{e_0,e_1,\ldots,e_s\}\subset E$ such that $W=\bigcap\limits_{i=1}^se_i$ and $e\cap W=\emptyset$ for all $e\in E\setminus E_0$,
       \item[(ii)] $|W|\ge 2$ and $W=\{v_0,v_1,\ldots ,v_k\}$,
       \item[(iii)] $\delta_V(v)=c$ for all $v\in W$, and  $\sum\limits_{e\in E_0}\frac{\delta_E(e)}{c}\frac{1}{|e|^2}=\nu$.
   \end{itemize}
   then $-\nu$ is an eigenvalue of $A$ with multiplicity $|W|-1$.
   \end{thm}
   \begin{proof}
   For each $i=1,2,\ldots,k$ we define a function $y_i\in \mathbb{R}^V$ as 
    \[
   y_i(v)=
   \begin{cases}
   1 &\text{~if~} v=v_0,\\
   -1&\text{~if~} v=v_i,\\
   0 &\text{~otherwise.~}
   \end{cases}
   \]
    By \Cref{A} we have $(Ay_i)(v)=\sum\limits_{e\in E_v}\frac{\delta_E(e)}{\delta_V(v)}\frac{1}{|e|^2}\sum\limits_{u\in e;u\neq v}y_i(u)$.
   Considering different cases we have the following facts.
   \begin{enumerate}
       \item[(a)] From condition(i) of the theorem we have $E_v=E_0 $  for all $v\in W$. Therefore, $(Ay_i)(v_i)
       =\sum\limits_{e\in E_0}\frac{\delta_E(e_i)}{\delta_V(u_i)}\frac{1}{|e_i|^2}y_i(v_0)=-\sum\limits_{e\in E_0}\frac{\delta_E(e_i)}{\delta_V(u_i)}\frac{1}{|e_i|^2}y_i(u_i)
       $ $=-\nu y_i(v_i)$
          and \\
       $(Ay_i)(v_0)=\sum\limits_{e\in E_0}\frac{\delta_E(e)}{\delta_V(v_0)}\frac{1}{|e|^2}\sum\limits_{u\in e;u\neq v_0}y_i(u)=\sum\limits_{e\in E_0}\frac{\delta_E(e_i)}{c}\frac{1}{|e_i|^2}y_i(v_i)=-\sum\limits_{e\in E_0}\frac{\delta_E(e_i)}{c}\frac{1}{|e_i|^2}y_i(u_i)=-\nu y_i(v_0)$.
       \item[(b)] For $i \ne j\ne 0$, we have $\sum\limits_{u\in e;u\ne v_j}y_i(u)=0$ for all $e\in E$ and thus $(Ay_i)(v_j)=0 $.
       \item[(c)] Note that $y_i(v)=0$ for all $v\in V\setminus W$ and for any $e\in E$, either both $v_0,v_i$ belongs to $e$ or none of them belongs to $e$. Therefore, $\sum\limits_{u\in e;u\ne v_j}y_i(u)=0$ for all $e\in E$ and this leads us to $(Ay_i)(v)=0 $ for all $v\notin W$.
   \end{enumerate}
   It is clear  that $-\nu$ is an eigenvalue of $A$. Since, $\{y_i\}_{i=1}^k$ is a linearly independent set, the rest of the proof is similarly as  in the proof of \Cref{cute2}. 
   \end{proof}
   \begin{thm}\label{adj_cute3}
Suppose that $G=(V,E)$ be a hypergraph with $E_0=\{e_0,e_1,\ldots,e_k\}\subset E$ such that 
\begin{itemize}
    \item [(1)] $W=
    \bigcap\limits_{i=0}^ke_i\neq \emptyset$, and $W\cap e=\emptyset$ for all $e\in E\setminus E_0$,
    \item[(2)]for all $i=0,1,\ldots,k$ there exists an $F_i\subset V$ such that $e_i=W\cup F_i$ with $|F_i|=t$ and $F_i\cap W=\emptyset$ for all $i$,
    \item[(3)] $F_i\cap e=\emptyset$ for all $e(\ne e_i)\in E$,
    \item[(4)]there exists $c,\omega\in \mathbb{R}$ such that $\delta_V(v)=c$ for all $v\in \bigcup\limits_{i=0}^kF_i$ and $\frac{\delta_E(e)}{|e|^2}=\omega$ for all $e\in E_0$.
\end{itemize}
 then $\frac{\omega}{c}(t-1)$ is an eigenvalue of $A$ with multiplicity at least $|E_0|-1$.
\end{thm}
\begin{proof}
We define $y_i\in \mathbb{R}^V$ for all $i=1,2,\ldots,k$, as
$$
y_i(v)=
\begin{cases}
-1&\text{~if~} v\in F_0\\
\phantom{-}1&\text{~if~}v\in F_i\\
\phantom{-}0&\text{~otherwise.~}
\end{cases}
$$
 By \Cref{A} we have $(Ay_i)(v)=\sum\limits_{e\in E_v}\frac{\delta_E(e)}{\delta_V(v)}\frac{1}{|e|^2}\sum\limits_{u\in e;u\neq v}y_i(u)$. Now we consider the following cases to prove the result.
\begin{itemize}
    \item[(a)]Since $E_v=\{e_j\}$, for $v\in F_j$,  $(Ay_i)(v)=\frac{\delta_E(e_j)}{\delta_V(v)}\frac{1}{|e_j|^2}\sum\limits_{u\in e;u\neq v}y_i(u)=\frac{\delta_E(e_j)}{\delta_V(v)}\frac{1}{|e_j|^2}(|F_j|-1)y_i(v)$. 
    \item[(b)] Since $E_v=E_0$ for all $v\in W$,  $(Ay_i)(v)=\sum\limits_{e\in E_0}\frac{\delta_E(e)}{\delta_V(v)}\frac{1}{|e|^2}\sum\limits_{u\in e;u\neq v}y_i(u)=\frac{\delta_E(e_0)}{\delta_V(v)}\frac{1}{|e_0|^2}(|F_0|-1)(-1)+\frac{\delta_E(e_i)}{\delta_V(v)}\frac{1}{|e_i|^2}(|F_i|-1)(1)=\frac{\omega}{c}(1-1)=0$.
    \item[(c)]For any $v\in V\setminus (W\cup(\bigcup\limits_{i=0}^kF_i))$, we have $\sum\limits_{u\in e;u\neq v}y_i(u)=0$ for all $e\in E_v$. Therefore, $(Ly_i)(v)=0.$ 
\end{itemize}
Thus $\frac{\omega}{c}(t-1)$ is an eigenvalue of $A$. Since $\{y_i\}_{i=1}^k$ are linearly independent, the multiplicity of the eigenvalue is at least $k=|E_0|-1$. 
\end{proof}


\subsection{Complete Adjacency Spectra of  Hyperflowers}
 Here we compute the complete list of eigenvalues of the adjacency operator associated with the $(l,1)$-hyperflower $G=(V,E)$ with $t$-twins. 
Suppose that for some $\gamma\in\mathbb{R} $, the function  $y_c\in\mathbb{R}^V$ is defined by
    $$ y_{\gamma}(v)=
    \begin{cases}
    \gamma &\text{~if~} v\in W,\\
    1 &\text{~if~} v\in U,
    \end{cases}
    $$where, $V=U\cup W$ is the partition of the set of vertices, as described in \Cref{hyperflower}. If $\frac{\delta_E(e)}{\delta_V(v)|e|^2}=\alpha$ for all $v\in V$ and for all $e\in E$ then
      $$ (A y_c)(v)=
    \begin{cases}
    l\alpha(\gamma(|W|-1)+t) &\text{~if~} v\in W,\\
    \alpha (|W|\gamma+(t-1) ) &\text{~if~} v\in U.
    \end{cases}
    $$
    Therefore, if $\gamma$ is a root of 
    \begin{equation}\label{hyperflower_last}
        |W|x^2+(t+l-l|W|-1)x-lt=0
    \end{equation}
    then $ y_\gamma$ is an eigenvector of $A$ with eigenvalue $\alpha (|W|\gamma+(t-1) )$. The two roots of \Cref{hyperflower_last} is going to give us two eigenvalues of $A$.
    Now by \Cref{adj_cute_1}, If $\alpha=\frac{\delta_E(e)}{\delta_V(v)|e|^2}$ for all $v\in V$ and for all $e\in E$ then corresponding to  $t$ twins of each hyperedge $e\in E$, $-\alpha$ becomes eigenvalue of $A$ with multiplicity at least $t-1$, where $\delta_V(v)=c$, for all $v\in e$. Evidently, for $l$ hyperedges, there are total $l(t-1)$ eigenvalues at least.
    If $\delta_V(v)=c$ for all $v\in U$ and $\delta_E(e)=\frac{\delta_E(e)}{|e|^2}=\mu$ for all $e\in E$ then  \Cref{adj_cute3} implies that $\frac{\mu}{c}
    (t-1)$ becomes an eigenvalue of $A$ with multiplicity $(l-1)$.
    Similarly, if $\delta_V(v)=c$ for all $v\in V$ and $\sum\limits_{e\in E}\frac{\delta_E(e)}{c|e|^2}=\nu$,  \Cref{adj_cute_2} concludes that $-\nu$ is an eigenvalue of $A$ with multiplicity $|W|-1$.
     Since, $(2+l(t-1)+(l-1)+|W|-1)=|V|$, thus we have the complete list of eigenvalues of $A$.
     Evidently, if $\frac{\delta_E(e)}{\delta_V(v)|e|^2}=\alpha$ for all $v\in V$, and for all $e\in E$ then  we have the determinant of $A_G$,
\begin{align*}
    \det(A_G)&=(-1)^{|V|-l-1}\left[(t-1)^2-|W|lt-(t-1)(t-1-|W|l+l)\right]\alpha^{|V|}(t-1)^{(l-1)}|E|^{|W|-1}\\
    &=(-1)^{|V|-l-1}l[1-(t+w)]\alpha^{|V|}(t-1)^{(l-1)}|E|^{|W|-1}.
\end{align*}
Note that, if $\alpha$ is an integer then $\det(A)$ is always an integer. For example, if we consider $\delta_V(v)=1$ and  $\delta_{E(e)}=|e|^2$ which implies the determinant of the adjacency matrix considered in \cite{rodriguez2003Laplacian, rodriguez2009Laplacian} is integer.

   Now we discuss some results, involving the the adjacency operator $A$. 
   \begin{rem}
\label{adjresult}   
  
    \begin{itemize}
        \item[(1)] Clearly, $r(v)=c$(constant) for all $v\in V$ then $\mathbf{1}$ is an eigenvector of $A$ with eigenvalue $c$.
        \item[(2)] For any $x,y\in\mathbb{R}^V$,  
        \begin{equation}\label{Axy}
            (Ax,y)_V= \sum\limits_
        {\begin{subarray}{1}
        {u,v\in V;}\\
        {u\neq v}
        \end{subarray}}x(u)y(v)\sum\limits_{e\in E_u\cap E_v}\frac{\delta_E(e)}{|e|^2}. \end{equation}
         Thus from  \Cref{Axy} we have $(Ax,y)_V=(Ay,x)_V=(x,Ay)_V$. So $A$ is a self-adjoint operator.
       
        \item[(3)] Corresponding to each $v\in V$, we define $\chi_v\in\mathbb{R}^V$ as 
        $$\chi_v(u)=
        \begin{cases}
        1&\text{~if~} u=v\\
        0&\text{~otherwise.~}
        \end{cases} $$
        Thus,
         $(A\chi_u,\chi_v)_V=\sum\limits_{e
        \in E_v\cap E_v}\frac{\delta_E(e)}{|e|^2}$. Therefore, the operator $A$ induces a matrix $B=\left(B_{uv}\right)_{u,v\in V}$ of order $|V|$ defined by 
        $$B_{uv}:=
        \begin{cases}
        (A\chi_u,\chi_v)_V=\sum\limits_{e
        \in E_v\cap E_v}\frac{\delta_E(e)}{|e|^2} &\text{~if~} u\neq v,\\
        0& \text{~otherwise.}
        \end{cases}$$
        \item[(4)] Since $A$ is self-adjoint, $B$ is symmetric matrix. Now form \Cref{A} we have $  (A_Gx)(v)
        =\sum\limits_{u\in V}\frac{1}{\delta_V(v)}\sum\limits_{e
        \in E_v\cap E_v}\frac{\delta_E(e)}{|e|^2}x(u)=\sum\limits_{u\in V}\frac{1}{\delta_V(v)}b_{vu}x(u)=\frac{1}{\delta_V(v)}(Bx)(v)$. Thus for the pre-assigned inner product $(\cdot,\cdot)_V$ on $\mathbb{R}$, the matrix $B$ can be directly deduced from the general adjacency operator $A$. From now onward we refer $B_G$ (or simply $B$) as the induced adjacency matrix associated withthe hypergraph $G$.  
         \item[(5)]\begin{itemize}
             \item[(a)] If $ \delta_E(e)=|e|^2$, then the matrix $B$ becomes the adjacency matrix of hypergraph given in \cite{rodriguez2003Laplacian,rodriguez2009Laplacian}.
             \item [(b)] If $ \delta_E(e)=\frac{|e|^2}{|e|-1}$, then the matrix $B$ coincides with the concept of adjacency matrix of hypergraph introduced in \cite{MR4208993}.
         \end{itemize} This fact motivates us to  incorporate the techniques used in \cite{MR4208993} on the matrix $B$.
         \item[(6)] Suppose that $\mathfrak{P}^n_{uv}$ is the set of all path of length $n$ connecting $u,v\in V$. For all \sloppy $p=ue_{1}v_{1}e_{2}\ldots e_{n}v\in \mathfrak{P}^n_{uv}$, we define $\mathfrak{E}(p):=\prod\limits_{i=1}^n\frac{\delta_E(e_i)}{|e_i|^2}$. Thus the $uv$-th entry of the matrix $B^n$ is $ B^n_{uv}=\sum\limits_{p\in \mathfrak{P}^n_{uv}}\mathfrak{E}(p)>0$ if and only if there exists a path $p\in \mathfrak{P}^n_{uv}$.
         \item[(7)]  The matrix $B$ induces an $1, 0$-matrix $B_0$ defined by ${B_0}_{uv}=0$ if $B_{uv}=0$ and otherwise ${B_0}_{uv}=1$. So $B_0$ is  the adjacency matrix of an unweighted graph $G_0=(V,E_0)$ defined by, for $u,v\in V$ with $u\neq v$, there exists an edge $\{u,v\}\in E_0$ if and only if there exists at least one hyperedge $e\in E$ such that $u,v\in e$. The hypergraph $G$ and the the graph $G_0$ have similar properties like, connectivity, graph colouring, etc. Moreover, if we impose an weight $ w_0:E_0\to\mathbb{R}$ on $G_0$, where $w_0$ is defined by $w_0(\{u,v\})=B_{uv}$, then the adjacency matrix $B$ of the hypergraph $G$ is also the adjacency matrix of the weighted graph $G_w=(V,E_0,w_0)$.
         \item [(8)] If there exists $u,v\in V$ such that the distance $d(u,v)=l$ then the $(u,v)$-th entry of the matrix $B^l$ is non-zero whereas the same for $I,B,B^2,\ldots,B^{l-1}$ are zero. Thus $I,B,B^2,\ldots,B^l$ are linearly independent. Similarly, if $diam(G)=k$, then $I,B,B^2,\ldots,B^k$ are linearly independent. If there exists $r$ distinct eigenvalues of $B$ then the degree of the minimal polynomial of $B$ is $r$. Thus there exists $c_0,c_1,\ldots, c_r\in \mathbb{R}$, not all zero, such that $c
         _0I+c_1B+c_2B^2+\ldots+c_rB^r=0$. Thus  $I,B,B^2,\ldots,B^r$ are linearly dependent and which implies $k \le r$, i.e.,
         the diameter of the hypergraph $G$ is less than the number of distinct eigenvalues of $B$. 
         \item[(9)] Since $B$ is a symmetric matrix, the result in \cite[Theorem 2.2]{MR4208993} can be restated as follows. 
         For a connected hypergraph $G(V,E)$ with $n$ vertices and minimum edge carnality 3, the diameter of $G$  $ diam(G) \le \bigg\lfloor 1 + \frac{\log((1-\alpha^2)/\alpha^2)}{\log(\lambda_{max}/\omega)} \bigg\rfloor,$
 where $\omega$  is the second largest eigenvalue (in absolute value) of $B$, $\lambda_{max}$ is the largest eigenvalue of  $B$ with the unit eigenvector $X_1=((X_1)_1, (X_1)_2,\dots, (X_1)_n)^t$, and $\alpha= \min_i \{(X_1)_i$\}.
         
    \end{itemize}
     \end{rem}

     \section{Normalized Laplacian operator}\label{nor-lap} In \Cref{gen},  we have mentioned that many conventional concepts of graph and hypergraph Laplacians, respectively, are actually special cases of the generalized Laplacian operator $\mathfrak{L}$. However, this generalized Laplacian fails to 
     represent some symmetrically normalized Laplacians, for example, the normalized Laplacian of hypergraphs  in \cite[section-4, Equation-16]{MR4208993} and the Laplacian given in \cite[section-1.2]{chung1997spectral}. In this section, we  introduce and study a general normalized Laplacian $\Tilde{\mathfrak{L}}$ for hypergraphs. Suppose that $\gamma:\mathbb{R}^V\to\mathbb{R}^V$ is an operator defined by $(\gamma(x))(v)=(r(v))^{-\frac{1}{2}}x(v)$. We define the general normalized Laplacian operator $\Tilde{\mathfrak{L}}:\mathbb{R}^V\to\mathbb{R}^V$ as $$\Tilde{\mathfrak{L}}=\gamma\circ\mathfrak{L}\circ \gamma. $$ Now we have some observations on $\Tilde{\mathfrak{L}}$.
     \begin{enumerate}
         \item For all $x\in \mathbb{R}^V$ and $v\in V$ we have, $(\Tilde{\mathfrak{L}}(x))(v)=x(v)-\sum\limits_{e\in E_v}\frac{\delta_E(e)}{\delta_V(v)}\frac{1}{|e|^2}\sum\limits_{u\in e;u\neq v}(r(u)r(v))^{-\frac{1}{2}}x(u)$.
         \item Evidently, $0$ is an eigenvalue of $\Tilde{\mathfrak{L}} $ and the dimension of eigenspace of $0$ is the number of connected components in the hypergraph. The function $\gamma(\mathbf{1})\in \mathbb{R}^V$, is an eigenvector, belongs to the eigenspace of $0$.
         \item Since $(\Tilde{\mathfrak{L}}x,x)_V=\sum_{e\in E}\frac{\delta_E(e)}{|e|^2}\sum\limits_{\{u,v\}\subset E}(\gamma(x)(u)-\gamma(x)(v))^2$, the operator $\Tilde{\mathfrak{L}}$ is positive semidefinite. Therefore, we have
         \begin{equation}\label{normalquard}
             (\Tilde{\mathfrak{L}}x,x)_V\le 2(x,x)_V.
         \end{equation}
         \item For $\delta_E(e)=\frac{|e|^2}{|e|-1}$ and $\delta_V(v)=1$, the operator $ \Tilde{\mathfrak{L}}$ becomes the normalized Laplacian operator described in  \cite[Section-4, Equation-15]{MR4208993}. If the hypergraph is a graph, $\Tilde{\mathfrak{L}} $ becomes the Laplacian given in \cite[Section-1.2]{chung1997spectral}.
         \item Consider the matrix $M=(M_{uv})_{u,v\in V}$ defined by $$M_{uv}=
         \sum_{e\in E_u\cap E_v}\frac{\delta_E(e)}{\delta_V(v)}\frac{|e|-1}{|e|^2} (r(u) r(v))^{-\frac{1}{2}}.
         $$ Evidently, $\Tilde{\mathfrak{L}}(x) =(I_{|V|}-M)(x) $. Therefore, if $\mu_1\le\mu_2\le\ldots\le\mu_{|V|}$ are the eigenvalues of $ \Tilde{\mathfrak{L}}$ then  the following holds.
         \begin{enumerate}
             \item If the hypergraph $G$ has no isolated vertex, then  $\sum\limits_{i=1}^{|V|}\mu_i=|V|$,
             \item Since $\mu_1=0$, we have $\mu_2\le\frac{|V|}{|V|-1}\le \mu_{|V|}$,
             \item \Cref{normalquard} leads us to
             $\mu_i\le 2$ for all $i=1,2,\ldots, |V|$.
          \end{enumerate}
     \end{enumerate}
\section{Applications} \label{app} Now we focus on the applications of the connectivity operators introduced in this work. In this section we study some application of our work in some conventional abstract classes of hypergraphs and some real-world situations.
 Use of the different Laplacian matrices associated with graphs in discrete dynamical network, diffusion, synchronization, random walk, image processing are common in literature, see  \cite{MR4079051,MR3730470,banerjee2020synchronization,MR1076116,MR1877614,wobrock2019image} and references therein. However,
 use of a hypergraph in place of
  the underlying graph may lead to better result sometimes. Instead of the conventional graph topology, some real-world networks need multi-body framework for better explanation. Indeed, incorporating hypergraph in proper way can accomplish the need of multi-body framework in many real-world phenomena. 

\subsection{Spectra of the Power of a Graph} Suppose that $G(V,E)$ is a graph, i.e., a $2$-uniform hypergraph. For any $k(\ge 3)\in \mathbb{N}$, the $k$-th power of $G$, denoted by $G^k=(U, F)$ is a $k$-uniform hypergraph, defined by $$U=V\cup \{\bigcup_{e\in E}W^k_e\} \text{~where,~}W^k_e=\{v_{ei}:i\in\mathbb{N},i\le k-2\}, \text{~and~} F=\{e^{(k)}=e\cup W^k_e:e\in E\}.$$
(See \cite[Definition 2.4]{MR3116407} for more details about the power of a graph).
In a graph, a vertex $v$ is said to be a pendant vertex if $|E_v|=1$. Suppose that $e$ is an edge of the graph $G$. Since $f\cap W^K_e=\emptyset$ for all $f(\ne e^{(k)})\in F$, one can use \Cref{cute2} and \Cref{adj_cute_1} to determine eigenvalues of the Laplacian operator and adjacency operator of $G^k$. Thus, we have the following result.
\begin{prop}
Suppose that $G(V,E)$ is a graph ($2$-uniform hypergraph). For all $k\ge 4$ and $e\in E$, If $\delta_V(v)=c_e$ for all $v\in W_e^k$ then the eigenvalues of the Laplacian and adjacency matrix of $G^k$ are given below.
\begin{enumerate}
    \item  $\frac{\delta_E(e)}{c_e}\frac{1}{k}$ is an eigenvalue of the Laplacian operator associated to $G^k$ with multiplicity $k-3$.
   \item $-\frac{\delta_E(e^{(k)})}{c_ek^2}$ is an eigenvalue of the general adjacency matrix of multiplicity $k-3$.
\end{enumerate}
If $e(\in E)$ contains a pendant vertex then instead of $k-3$, in the above two cases, the multiplicity becomes $k-2$.
\end{prop}
As we have done before, 
here,  we can also compute the eigenvalues in a particular framework by choosing $\delta_E,\delta_V$ appropriately. 
\subsection{Spectra of Squid}
A squid is a $k$-uniform hypergraph $G(V,E)$ such that 
$$V:=\{v_0\}\cup(\bigcup\limits_{i=1}^{k-1}U_i)\text{~where,~} U_i=\{u_{ij}:j\in \mathbb{N}, 1\le j\le k\},$$
$$\text{~and~}E:=\{U_i\}_{i=1}^{k-1}\cup \{\{v_0\}\cup e_0\} \text{~where,~} e_0=\{u_{i1}:1\le i\le k-1\}.$$
We consider $\{v_0\}\cup e_0$ as a central hyperedge and all other hyperedge of squid as peripheral hyperedges (see \cite{MR3116407} for more details about squid).
Since, $e\cap (U_{i}\setminus\{u_{i1}\})=\emptyset$ for all $e(\ne U_i)\in E$, therefore, using \Cref{cute2} and \Cref{adj_cute_1} we have the following result. 
\begin{prop}
Suppose that $G(V,E)$ is a $k$-uniform squid. For any peripheral hyperedge $U_i$, if $ \delta_V(v)=c_i$ for all $v\in U_i$ then the eigenvalues of the general adjacency and Laplacian matrix of the squid is given below.
\begin{enumerate}
    \item  $\frac{\delta_E(U_i)}{c_i}\frac{1}{k}$ is an eigenvalue of the Laplacian operator associated to $G^k$ with multiplicity $k-2$.
   \item $-\frac{\delta_E(U_i)}{c_ik^2}$ is an eigenvalue of the general adjacency matrix of multiplicity $k-2$.
\end{enumerate}
\end{prop}

\subsection{The network of disease propagation}
Multi-body interactions are crucial to study disease propagation. In past few years, using of hypergraphs made the disease propagation models more realistic, see \cite{higham2021epidemics}. Here the vertices of the hypergraph $G(V,E)$ represent the individuals and hyperedges are the collection of individuals who are known to interact as a group.  We summarize below the applicability of our work in this context. 
    \begin{enumerate}
        \item If we set $\delta_E(e)=\beta |e|^2$ and $\delta_V(v)=1$ then according to the general infection model provided in \cite[p.6 , Section-3.2.]{higham2021epidemics}, a susceptible node $v$ becomes infectious with the rate $(A_G(\bar{f}(x_t)))(v)$. Here, $x:V\times T
               \to \mathbb{R}^+ $ is a function where $T$ is the domain of time and for any $(v,t)\in V\times T $, the functional value of $x(v,t)$ is denoted by $x_t(v)$. That is, $x_t\in{\mathbb{R}^+}^V$ is defined as $x_t(v):=x(v,t)$. 
               In addition, the function $f:\mathbb{R}^+\to \mathbb{R}^+$ regulates the overall  infectiousness of the disease and $\bar{f}:{\mathbb{R}^+}^V\to {\mathbb{R}^+}^V$ is defined as $\bar{f}(x)=\{f(x(v))\}_{v\in V}$.  Similar infection rate is also reported  in \cite{MR3494570}. Later in partitioned hypergraph model \cite[p.6 , Section-3.3.]{higham2021epidemics}, the hypergraph $G(V,E)$ is partitioned in to $K$ disjoint hypergraphs $\{G_i(V_i,E_i)\}_{i=1}^K$. According to this model the infection rate of the node $v$ at time $t$ is  $\sum\limits_{i=1}^KA_{G_i}(\bar{f_i}(x_t)))(v)$, where the function $f_i:\mathbb{R}^+\to \mathbb{R}^+$ regulates the overall  infectiousness of the disease in the $i$-th partition.
               \item To study random infection rates, in \cite[p.6 , Section-4.]{higham2021epidemics}, the mean field approximation is considered. According to that approach, the infection rate of a node $v$ at time $t$ is $(A_G(\bar{f}(P_t)))(v)$, where $p_t(v)$ is the probability of being the node $v$ is infected at time $t$ and $P_t:=\{p_t(v)\}_{v\in V}$.
    \end{enumerate}

\subsection{Dynamical network}A \textit{dynamical network} is a network of evolving \textit{dynamical systems}. More precisely, a dynamical system is a system in which a function describes the evolution of a point in a geometric space with the flow of time. In a dynamical network, several dynamical systems are coupled through an underlying network in such a way that two neighbouring dynamical systems influence the dynamics of each other. The underlying network may be a graph\cite{MR3730470} or a hypergraph\cite{banerjee2020synchronization,MR4121260,carletti2020dynamical}. To discuss  coupled dynamics on hypergraphs the adjacency operator $A_G$ is used in \cite[equation-(24),(27)]{MR4121260} with $\delta_E(e)=|e|^2 $ and $\delta_V(v)=1$. In \cite{banerjee2020synchronization}, the diffusion operator $L_G$ is used with $\delta_E(e)=w(e)\frac{|e|^2}{|e|-1} $ in order to discuss synchronization in dynamical networks on hypergraph. In \cite[Equation-3]{carletti2020dynamical}, one variant of the general Laplacian operator of hypergraph, $\mathfrak{L}$ is used in the model of dynamical systems on hypergraphs with $\delta_E(e)=(|e|-1)|e|^2$ and $\delta_V(v)=1$. Considering the use of different variant of the diffusion operator $L_G$ in distinct dynamical networks  with hypergraph topology, we  define a general discrete dynamical network model as
\begin{equation}
    x_{n+1}=f(x_{n})+\epsilon (L_G(g(x_n))),
\end{equation}
where for any discrete time $n\in \mathbb{N}$, $x_n\in (\mathbb{R}^V$) is a function such that $x_n(v)$ is the state of the $n$-th node. Both $f:\mathbb{R}^V\to \mathbb{R}^V$ and $g:\mathbb{R}^V\to \mathbb{R}^V$ are differentiable functions, regulating the dynamics of all the node. The positive real $\epsilon$ is the coupling strength. Similarly, the continuous model can be defined as 
\begin{equation}
    \dot{x}_t=f(x_t)+\epsilon (L_G(g(x_t))),
\end{equation}
where $x_t\in \mathbb{R}^V$ is such that $x_t(v)$ is the state of the $v$-th node at time $t$ and $\dot{x}\in \mathbb{R}^V$ is defined by $\dot{x_t}(v)=\frac{dx_t(v)}{dt} $.
\subsection{Random walk on hypergraphs}
A random walk is a sequence of randomly taken successive steps by a walker in a mathematical space. If the mathematical space is the set of all the vertices $V$ of a hypergraph $G(V,E)$ then the random walk is referred as the random walk on the hypergraph. Thus, a random walk on a hypergraph $G(V,E)$ is a sequence of vertices $v_1,v_2,\ldots, v_k$ such that $v_i $ is the $i $-th step of the random walk. The whole theory pivot around the \textit{Transition probability}, $(P_G)_{uv}=prob(v_{i+1}=v|v_i=u)$, which is independent of $i$ and depends on the underlying hypergraph. Since, $\bigcup\limits_{v\in V}\{(v_{i+1}=v|v_i=u)\} $ is a certain event, $\sum\limits_{v\in V}(P_G)_{uv}=1$ for all $u\in V$. We  define ${P_G}_{uv}$ as 
$${P_G}_{uv}=
\begin{cases}
\frac{1}{r(u)}\sum\limits_{e\in E_u\cap E_v}\frac{\delta_E(e)}{\delta_V(u)}\frac{1}{|e|^2} & \text{~if~} u\neq v,\\
0& \text{~otherwise.~}
\end{cases} $$
We summarise below some crucial observations.
\begin{itemize}
    \item[(1)]Since, $ \sum\limits_{v\in V}\sum\limits_{e\in E_u\cap E_v}\frac{\delta_E(e)}{\delta_V(u)}\frac{1}{|e|^2}=\sum\limits_{e\in E_u}\frac{\delta_E(e)}{\delta_V(u)}\frac{|e|-1}{|e|^2}=r(u)$, we have $\sum\limits_{v\in V}(P_G)_{uv}=1$. 
    \item[(2)]  Suppose that there exists no isolated vertex in $G$, i.e.,  $E_v\neq \emptyset$ for all $v\in V$. So, $r(v)\neq 0$ for all $v\in V$ and this allow us to define the inner product $(\cdot ,\cdot)_R$ on $\mathbb{R}^V$ as $(x,y)_R:=\sum\limits_{u\in V}r(u)\delta_V(u)x(u)y(u)$. If $\Delta=\mathcal{I}-P_G$, where $\mathcal{I}:\mathbb{R}^V\to \mathbb{R}^V$ is the identity operator on $\mathbb{R}^V$, then $0$ is an eigenvalue of $\Delta$ with eigenvector $\mathbf{1}$.  Moreover, $(\Delta x, y)_R=\sum\limits_{\{u,v\}\subset V}\sum\limits_{e\in E_u\cap E_v}\frac{\delta_E(e)}{|e|^2}(x(u)-x(v))^2\le 2(x,y)_R$. Therefore, $\Delta$ is a positive semidefinite operator and all the eigenvalues of $\Delta$ lies in $[0,2) $. Thus, all the absolute values of all the eigenvalues of $P_G$ lies in $[0,1]$. Moreover, if the hypergraph $G$ is connected then except the eigenvalue $0$ corresponding to the eigenvector $\mathbf{1}$, the absolute value of all other eigenvalues of $P_G$ lie in $(0,1)$.  
    \item[(3)]Note that $ (\Delta x, y)_R=\sum\limits_{\{u,v\}\subset V}\sum\limits_{e\in E_u\cap E_v}\frac{\delta_E(e)}{|e|^2}(x(u)-x(v))^2=(\Delta y,x)_R=(x,\Delta y)_R$. Thus, $\Delta $ is self-adjoint. Therefore, $P_G$ is also self-adjoint.
    \item[(4)] Suppose that $\{x_n\}_{n\in\mathbb{N}}$ is a sequence in $\mathbb{R}^V$ such that $x_{n+1}=P_G(x_n)$ and the underlying hypergraph is connected. Evidently, $x_{n+1}=P_G^n(x_1)$. Since except the eigenvalue $0$ corresponding to the eigenvector $\mathbf{1}$, the absolute value of all the eigenvalues of $P_G$ lie in $(0,1)$, by spectral decomposition, $\lim\limits_{n\to\infty}x_n$ is the projection of the initial state  $x_1$ along the vector $\mathbf{1}$. Therefore, $\lim\limits_{n\to\infty}x_n=\frac{(x_1,\mathbf{1})_R}{\sqrt{(\mathbf{1},\mathbf{1})_R}}\mathbf{1}$.
    
    Note that, the properties of  general normalized Laplacian operator $\Tilde{\mathfrak{L}}$ suggest that we can replace $\Delta$ by $\Tilde{\mathfrak{L}}$.
\end{itemize}
We end this article with the following Remark.
\begin{rem}
Since $\delta_V \in {\mathbb R^+}^V$ and $\delta_E \in {\mathbb R^+}^E$, there exists uncountable choices for $\delta_E,\delta_V$. Each choice is going to give us a framework for the operators associated to a hypergraph. Although some results (see \Cref{cute-new}, \Cref{cute1}, \Cref{adj_cute_2}, \Cref{cute3}) imposes such conditions on $\delta_V$, that very few choices left for $\delta_V$ but since very few conditions are imposed on $\delta_E$, one still has uncountable choices for $\delta_E$. Therefore, our results are valid for uncountable number of frameworks of operators. Two of these frameworks are common in literature and considered in \cite{MR4208993} and \cite{rodriguez2003Laplacian,rodriguez2009Laplacian,bretto2013hypergraph}.
\end{rem}
\section*{Acknowledgement}The work of the author PARUI is supported by University Grants Commission, India (Beneficiary Code/Flag: 	BININ00965055 A). PARUI is sincerely thankful to Rajiv Mishra, Gargi Ghosh for fruitful discussions. 
	\bibliographystyle{siam}
	
\end{document}

%% file: black.tikzstyles
\tikzstyle{black}=[fill=white, draw=black, shape=circle, tikzit fill=white, tikzit draw=black, tikzit shape=circle]
\tikzstyle{blackfilled}=[fill={rgb,255: red,50; green,50; blue,50}, draw=black, shape=circle, tikzit fill={rgb,255: red,50; green,50; blue,50}, tikzit draw=black, tikzit shape=circle]
